\documentclass[12pt]{amsart}
\RequirePackage[dvipsnames,usenames]{color}
\usepackage{amssymb,amsthm,amsmath,amscd,quiver}
\pdfoutput=1
\usepackage{microtype}
\usepackage{mathtools}
\usepackage{soul}
\usepackage[backref=page, colorlinks=true, linkcolor=magenta, citecolor=cyan]{hyperref}
\usepackage{marvosym}
\usepackage{calligra}

\input{kmacros3.sty}

\usepackage{tikz}
\usepackage{tikz-cd}
\usetikzlibrary{cd}
\usepackage{graphicx}
\usepackage[all,cmtip]{xy}
\usepackage{mathpazo}

\usepackage{mabliautoref}
\numberwithin{equation}{theorem}

\usepackage{bm}
\usepackage{pifont}
\usepackage{upgreek}

\usepackage{eucal}
\usepackage{ulem}
\usepackage{stmaryrd}

\DeclareMathOperator{\perf}{perf}

\numberwithin{equation}{theorem}

\usepackage{fullpage}

\usepackage{setspace}
\usepackage{enumerate}
\usepackage{calc}

\usepackage{verbatim}
\usepackage{alltt}
\begin{document}

\title{When is Frobenius epic?}
\author{Javier Carvajal-Rojas}
\address{Centro de Investigaci\'on en Matem\'aticas, A.C., Callej\'on Jalisco s/n, 36023 Col. Valenciana, Guanajuato, Gto, M\'exico}
\email{\href{mailto:javier.carvajal@cimat.mx}{javier.carvajal@cimat.mx}}
\author{Rankeya Datta}
\address{University of Missouri, 219 Mathematical Sciences Building, 810 Rollins St., Columbia MO 65203, United States}
\email{\href{mailto:rankeya.datta@missouri.edu}{rankeya.datta@missouri.edu}}
\author{Noah Olander}
\address{UC Berkeley Mathematics Department, 970 Evans Hall, Berkeley, CA 94720, United States}
\email{\href{mailto:nolander@berkeley.edu}{nolander@berkeley.edu}}
\author{Axel St{\"a}bler}
\address{Universit\"at Leipzig\\ Mathematisches Institut\\ Augustusplatz 10\\
04109 Leipzig\\Germany} 
\email{\href{mailto:staebler@math.uni-leipzig.de}{staebler@math.uni-leipzig.de}}
\date{}


\thanks{Carvajal-Rojas was partially supported by the grants SECIHTI \#CBF2023-2024-224, \#CF-2023-G-33, and \#CBF-2025-I-673. Datta was supported in part by NSF DMS Grant \#2502333 and Simons Foundation grant
MP-TSM-00002400. Olander was supported in part by NSF DMS Grant \#2402087.}


\begin{abstract}
We prove that a homomorphism of rings of prime characteristic is b-nil formally unramified if and only if its relative Frobenius is an epimorphism. We specialize this result under different finiteness conditions such as (relative) $F$-finiteness and Noetherianity. We give an example where the absolute Frobenius is an epimorphism but is not surjective, as well as an example of a formally unramified homomorphism that is not b-nil formally unramified.
\end{abstract}

%\begin{document}
\maketitle

\section{Introduction}

In the early development of scheme theory, Grothendieck advocated the use of formal lifting properties to define classes of morphisms. For example, he defined smooth morphisms as formally smooth morphisms (in the discrete sense) that are locally of finite 
presentation \cite[D\'efinition 17.3.1]{EGAIV}. This definition is particularly useful for checking that a moduli space is smooth over the base, as it can be stated entirely in terms of the functor of points. Grothendieck also proved the deep result that formally smooth, but not necessarily finitely presented, maps of Noetherian rings are regular; that is, flat
with geometrically regular fibers (see \cite[Corollaire 19.6.5]{EGAIV}). However, in the absence of any finiteness hypotheses, there are many pathological examples of formally smooth, \'etale, and unramified ring maps. For example, a formally \'etale algebra over a field need not be reduced; see \cite{Bhatt_imperfect_trivial_cc}.

In the work \cite{DattaOlanderRelFrobIso}, the second and third named authors of this article investigate a strengthening of these lifting properties that was introduced by Berthelot in \cite[IV, D\'efinition 1.5.1]{berthelot1974cohomologie}. Briefly, instead of considering lifts of ring maps with respect to nilpotent ideals, one considers lifts with respect to ideals that are nil of bounded index; see \autoref{sec:main}. The resulting analogs of formally smooth, \'etale, and unramified ring maps are called b-nil formally smooth, \'etale, and unramified ring maps. Datta and Olander prove several structural properties of b-nil formally \'etale and b-nil formally smooth maps in prime characteristic $p>0$ without any finiteness hypotheses. Namely, b-nil formally smooth algebras over fields of positive characteristic are reduced; for b-nil formally smooth ring maps of $\mathbf{F}_p$-algebras, the relative Frobenius is flat, and it even exhibits the target as a projective module over the source; and a map of $\mathbf{F}_p$-algebras is b-nil formally \'etale if and only if its relative Frobenius is an isomorphism. The work loc. cit. leaves open the question of whether there is a characterization of b-nil formally unramified ring maps of $\mathbf{F}_p$-algebras in terms of the relative Frobenius.

 In the independent work \cite{CarvajalRojasStablerPristineMorphisms}, the first and fourth named authors develop the theory of flat morphisms of $\mathbf{F}_p$-schemes with isomorphic relative Frobenius, which they call pristine morphisms, and they show that this class of morphisms behaves well from the perspective of F-singularity theory of locally  Noetherian schemes. An important technical result in their paper is that a map of $\mathbf{F}_p$-algebras $R \to S$, for which both $S$ and its Frobenius twist $S^{(p)} = F_*R\otimes_R S$ are Noetherian, is formally unramified if and only if the relative Frobenius $F_{S/R} : S^{(p)} \to S$ is surjective \cite[Theorem 4.1]{CarvajalRojasStablerPristineMorphisms}. In this paper, we prove an analogous characterization of b-nil formally unramified morphisms in positive characteristic with no finiteness assumptions. 

\begin{theorem}(\autoref{thm:b-nil-formally-unramified})
\label{thm:main-intro}
     Let $\varphi \colon R \to S$ be a map of $\mathbf{F}_p$-algebras. The following are equivalent: 
    \begin{enumerate}
        \item $\varphi$ is b-nil formally unramified.
        \item $F_{S/R}$ is b-nil formally unramified.
        \item $F_{S/R}$ is an epimorphism of rings.
    \end{enumerate}
\end{theorem}

We give an example to show that the result does not hold when ``epimorphism'' is replaced by ``surjection.'' In fact, we construct an example of an $\mathbf{F}_p$-algebra $A$ whose absolute Frobenius is an epimorphism but not surjective; see \autoref{sec:epimorphic-Frobenius-not-surjective}. This answers a Mathoverflow question from Matthieu Romagny \cite{MOFrobeniusepisurjective} and shows that \autoref{thm:main-intro} is optimal because the absolute Frobenius of $A$ is the relative Frobenius of $\mathbf{F}_p \to A$. We provide another example to show that a formally unramified map of rings of prime characteristic $p>0$ need not be b-nil formally unramified; see \autoref{sec:formally-unramified-not-bnil}. Thus, b-nil formal unramifiedness is a stronger property than formal unramifiedness. In addition, we prove several strengthenings of \autoref{thm:main-intro} under various finiteness conditions. For instance, in \autoref{prop:b-nil-formally-unramified-image-relFrobenius-Noetherian}, we offer another perspective on the characterization of formally unramified maps of Noetherian $\mathbf{F}_p$-algebras that was recently given in \cite[Theorem 4.1]{CarvajalRojasStablerPristineMorphisms} by the first and fourth authors. We prove in \autoref{thm:F-finite-b-nil-formally-unramified} that if a homomorphism of prime characteristic rings has finite relative Frobenius, then it is formally unramified if and only if it is b-nil formally unramified. This is the analog of \cite[Theorem 1.5]{DattaOlanderRelFrobIso}, which states that if a homomorphism of prime characteristic rings has finitely presented relative Frobenius, then it is b-nil formally smooth (resp. \'etale) if and only if it is formally smooth (resp. \'etale). The thrust of these enhancements of \autoref{thm:main-intro} is that under finiteness conditions, epimorphisms are often surjective.

The paper is organized as follows. In \autoref{sec:epimorphisms}, we collect for later use some generalities on epimorphisms in the category of rings. This notion is surprisingly subtle and was the subject of a seminar by Samuel; see \cite{SamuelSeminar}. In \autoref{sec:main}, we give the proof of \autoref{thm:main-intro}. In \autoref{sec:noetherian}, we discuss consequences in the Noetherian case, and in particular, the relationship between \autoref{thm:main-intro} and \cite[Theorem 4.1]{CarvajalRojasStablerPristineMorphisms}. Finally, in \autoref{sec:epimorphic-Frobenius-not-surjective} and \autoref{sec:formally-unramified-not-bnil}, we give examples of, respectively, a ring of prime characteristic whose absolute Frobenius is an epimorphism but not surjective, and a ring map of prime characteristic that is formally unramified but not b-nil formally unramified. 

    \section{Preliminaries}
    When we talk about formal properties of ring maps in this paper, such as formally smooth/unramified/\'etale, we do so in the discrete sense. The terminology such as $0$-smooth/unramified/\'etale is also used in the literature \cite{MatsumuraCommutativeRingTheory}.

    If $R$ is a ring of prime characteristic $p > 0$, by $F_*R$ we mean that the ring $R$ is viewed as an $R$-algebra by restriction of scalars along the absolute Frobenius
    \begin{align*}
        F_R \colon R &\to R\\
        r&\mapsto r^p.
    \end{align*}
    We omit $R$ from the subscript of $F_R$ when it is clear from the context.
    For an ideal $I$ of $R$, $I^{[p]}$ is the extension of $I$ along $F$. That is, $I^{[p]}$ is the ideal of $R$ generated by the $p$-th powers of elements of $I$.

    If $\varphi \colon R \to S$ is a map of $\mathbf{F}_p$-algebras, then the \emph{relative Frobenius} $F_\varphi$ or $F_{S/R}$ is the unique ring map $F_*R \otimes_R S \to F_*S$ such that the following diagram commutes:
    \[
    \begin{tikzcd}[column sep=huge]
	R & S & \\
	{F_*R} & {F_*R\otimes_RS} \\
	&& {F_*S.}
	\arrow["\varphi", from=1-1, to=1-2]
	\arrow["{F_R}"', from=1-1, to=2-1]
	\arrow["{F_R\otimes_R\id_S}", from=1-2, to=2-2]
	\arrow["{F_S}", curve={height=-30pt}, from=1-2, to=3-3]
	\arrow["{\id_{F_*R}\otimes_R\varphi}"', from=2-1, to=2-2]
	\arrow["{F_*\varphi}"', curve={height=24pt}, from=2-1, to=3-3]
	\arrow["{F_\varphi}"', color={rgb,255:red,214;green,92;blue,92}, from=2-2, to=3-3]
\end{tikzcd}
\]
The domain, $F_*R \otimes_R S$, of $F_\varphi$ will often be denoted as $S^{(p)}$.

\section{Epimorphisms of Rings}
\label{sec:epimorphisms}

In this paper, we will analyze when the relative Frobenius of a map of $\mathbf{F}_p$-algebras is an epimorphism. So we recall some basic facts about ring epimorphisms in this section for the reader's convenience. The best reference is the expos\'es for Samuel's seminar on ring epimorphisms \cite{SamuelSeminar}. 

In the rest of the paper, for a ring map $\varphi \colon R \to S$, we use $\mu_\varphi$ or $\mu_{S/R}$ to denote the canonical multiplication map 
\begin{align*}
    S \otimes_R S &\twoheadrightarrow S\\
    x \otimes y &\mapsto xy.
\end{align*}
Note that $\mu_{\varphi}$ is surjective with kernel 
\begin{equation} \label{eqn.Definition_Diagonal_Ideal}
 \ker (\mu_{\varphi}) = \langle s \otimes 1 - 1 \otimes s \mid s\in S \rangle.
\end{equation}
%the ideal generated by elements $s \otimes 1 - 1 \otimes s$. $s \otimes 1 - 1 \otimes s$
In addition, $\mu_{\varphi}$ is a retraction for both of the canonical homomorphisms $S \to S \otimes_R S$. Furthermore, we have a split short exact sequence of $S$-modules (for each canonical map $S \to S \otimes_R S$)
\[
0 \to S \to S \otimes_R S \to S \otimes_R \coker(\varphi) \to 0
\]
which yields the direct sum decomposition
\[
S \otimes_R S = \im(S \to S \otimes_R S) \oplus \ker(\mu_{\varphi})
\]
and so an isomorphism of $S$-modules
\[
%\ker \mu_{\varphi}
\ker(\mu_{\varphi}) \cong S \otimes_R \coker(\varphi).
\]
From this, one can readily derive the following alternative characterizations of epimorphisms. 

\begin{lemma}
    \label{lem:epimorphism-alternative-characterizations}
    Let $\varphi \colon R \to S$ be a ring map. Then the following are equivalent:
    \begin{enumerate}
        \item $\varphi$ is an epimorphism (in the category of rings). \label{lem:epimorphism-alternative-characterizations.a}
        \item The two canonical maps $S \to S \otimes_R S$ are equal. \label{lem:epimorphism-alternative-characterizations.b}
        \item Either of the two canonical injections $S \to S \otimes_R S$ is an isomorphism. \label{lem:epimorphism-alternative-characterizations.c}
        \item $S \otimes_R \coker(\varphi) = 0$. \label{lem:epimorphism-alternative-characterizations.d}
        \item The multiplication map $\mu_\varphi \colon S \otimes_R S \twoheadrightarrow S$ is an isomorphism. \label{lem:epimorphism-alternative-characterizations.e}
    \end{enumerate}
\end{lemma}

\begin{comment}
\begin{proof}
    This is standard. Note that by right exactness of tensor products $S \otimes_R \coker(\varphi)$ is the cokernel of either of the canonical injections $S \to S \otimes_R S$.
\end{proof}
\end{comment}
\autoref{lem:epimorphism-alternative-characterizations} immediately implies the following two corollaries.

\begin{corollary}
    \label{cor:epimorphisms-formally-unramified}
    Let $\varphi \colon R \to S$ be a ring epimorphism.
    \begin{enumerate}
        \item The base change $\id_T \otimes_R \varphi \colon T \to T \otimes_R S$ is an epimorphism for all $R$-algebras $T$.

        \item $\varphi$ is formally unramified.

        \item For every $S$-module $M$, the canonical homomorphism $M \to M \otimes_R S$ is an isomorphism. In other words, every $S$-module is the base change of an $R$-module.
    \end{enumerate}
\end{corollary}

\begin{corollary}
    \label{cor:epic-upto-nilpotent-ideal}
    Let $\varphi \colon R \to S$ be a ring map and $I \subset R$ be a nilpotent ideal (i.e. $I^n = 0$ for some $n > 0$). If the induced map $\varphi \otimes_R \id_{R/I} \colon R/I \to S/IS$ is an epimorphism, then $\varphi$ is an epimorphism.
\end{corollary}

\begin{proof}
    We have $\coker(\varphi \otimes_R \id_{R/I}) = \coker(\varphi) \otimes_R R/I$. Thus, 
    \[
    S/IS \otimes_{R/I} \coker(\varphi \otimes_R \id_{R/I}) = (S \otimes_R \coker(\varphi)) \otimes_R R/I.
    \]
    %Since $\varphi \otimes_R \id_{R/I}$ is an epimorphism, $(S \otimes_R \coker(\varphi)) \otimes_R R/I = 0$
    Since $\varphi \otimes_R \id_{R/I}$ is an epimorphism, this vanishes by \autoref{lem:epimorphism-alternative-characterizations}. Therefore, by the nilpotent version of Nakayama's Lemma \cite[\href{https://stacks.math.columbia.edu/tag/00DV}{Tag~00DV~(9)}]{stacks-project}, $S \otimes_R \coker(\varphi) = 0$, and so $\varphi$ is an epimorphism by \autoref{lem:epimorphism-alternative-characterizations}.
\end{proof}

\begin{example}[Ring epimorphisms need not be surjective.]
    The canonical map associated with the localization of a ring at a multiplicative set is always an epimorphism, but it is usually not surjective. Thus, unlike in the category of modules, epimorphisms in the category of commutative rings are not always surjective.
\end{example}

Recall that a ring map $R \to S$ is \emph{universally injective (or pure)} if for all $R$-modules $M$, the induced map $M \to S \otimes_R M$ is injective. Taking $M = R$, we see that universally injective ring maps are injective. Faithfully flat ring maps are universally injective. If $R \to S$ admits an $R$-linear left inverse, then $R \to S$ is universally injective.

%\todo{Isn't universally injective better known as pure?}

Following \cite{rg71}, we introduce the following notion.

\begin{definition}
    \label{def:property-o}
    A ring map $R \to S$ satisfies \emph{property (O)} if for all $R$-modules $M$, whenever $S \otimes_R M = 0$, then $M = 0$.
\end{definition}

\begin{remark}
\label{rem:ring-maps-satisfying-O}
    Let $\varphi \colon R \to S$ be a ring map.
    \begin{enumerate}
        \item If $\varphi$ is universally injective, then $\varphi$ satisfies (O). This follows because $M \to S \otimes_R M$ is injective for all $R$-modules $M$.\label{rem:ring-maps-satisfying-O.a}
        \item More generally, suppose $\varphi$ \emph{descends flatness}. That is, suppose for all $R$-modules $M$, if $S \otimes_R M$ is $S$-flat, then $M$ is $R$-flat. If $\varphi$ is additionally injective, then it satisfies (O). Indeed, $S \otimes_R M = 0$ implies $M$ is $R$-flat. Then $M \to S \otimes_R M$ is injective by the injectivity of $\varphi$, and so $M = 0$. Note that a universally injective ring map is injective and descends flatness \cite[\href{https://stacks.math.columbia.edu/tag/08XD}{Tag 08XD}]{stacks-project}. If $\varphi$ is injective and $R$ is Noetherian, then $\varphi$ satisfies (O) if and only if $\varphi$ descends flatness by \cite[Part~II,~Corollaire~1.2.10]{rg71}.\label{rem:ring-maps-satisfying-O.b}
        \item Suppose $\varphi$ has the property that for all $R$-modules $M$, if $\Hom_R(S,M) = 0$, then $M = 0$. Then $\varphi$ satisfies (O) by \cite[Part~II,~Lemme~1.2.2]{rg71}.\label{rem:ring-maps-satisfying-O.c}
        \item Ring maps that are not injective can satisfy (O). For example, let $R$ be a non-reduced ring and let $I$ be a non-zero finitely generated nilpotent ideal of $R$ (i.e., $I^n = 0$ for some integer $n > 0$). Then $R \twoheadrightarrow R/I$ satisfies (O). Indeed, if $R/I \otimes_R M = 0$, then $M = IM$, and by induction, $M = I^nM = 0$.\label{rem:ring-maps-satisfying-O.d}
    \end{enumerate}
\end{remark}

\begin{lemma}
    \label{lem:O-epimorphism}
    Let $\varphi \colon R \to S$ be a ring epimorphism that satisfies (O). Then $\varphi$ is surjective and $\ker(\varphi)$ is contained in the nilradical of $R$. %Thus, if $\varphi$ is additionally injective (for e.g., if $\varphi$ is universally injective), then $\varphi$ is an isomorphism.
\end{lemma}

\begin{proof}
    By \autoref{lem:epimorphism-alternative-characterizations}~\autoref{lem:epimorphism-alternative-characterizations.d}, we have $S \otimes_R \coker(\varphi) = 0$; hence $\coker(\varphi) = 0$ proving surjectivity. Thus, we may replace $S$ by $R/\ker(\varphi)$. Then, for all prime ideals $\frp$ of $R$, $R/\ker(\varphi) \otimes_R R_\frp \neq 0$ as $R_\frp \neq 0$. However, this implies $\ker(\varphi) \subset \frp$.
\end{proof}

\begin{remark}
\label{rem:weaker-than-O}
    Surjectivity in the proof of \autoref{lem:O-epimorphism} follows if $S$ satisfies the weaker property that if $M$ is a quotient $R$-module of $S$ and $S \otimes_R M = 0$, then $M = 0$. This observation will be useful in \autoref{prop:epic-implies-surjective}.
\end{remark}

\begin{corollary}
    \label{cor:universally-injective-epimorphism}
    A universally injective epimorphism of rings is an isomorphism.
\end{corollary}

\begin{proof}
    This follows from \autoref{lem:O-epimorphism} as universally injective ring maps are injective ring maps that satisfy (O) by \autoref{rem:ring-maps-satisfying-O} (a).
\end{proof}

\begin{corollary}
    \label{cor:epimorphisms-from-fields}
    Let $\varphi \colon k \to R$ be a ring map where $k$ is a field. Then $\varphi$ is an epimorphism if and only if either $R = 0$ or $\varphi$ is an isomorphism.
\end{corollary}

\begin{proof}
    The interesting implication is $\implies$. If $R \neq 0$, then $\varphi \colon k \to R$ is split injective and hence universally injective. Then $\varphi$ is an isomorphism by \autoref{cor:universally-injective-epimorphism}.
\end{proof}

\begin{remark}
\label{rem:injective-on-spec}
    Since epimorphisms are preserved under arbitrary base change (\autoref{cor:epimorphisms-formally-unramified}), if $R \to S$ is an epimorphism, then \autoref{cor:epimorphisms-from-fields} implies that for all $\frp \in \Spec(R)$, $\kappa(\frp) \otimes_R S$ is either the zero ring or isomorphic to $\kappa(\frp)$. Thus, the induced map $\Spec(S) \to \Spec(R)$ is injective and induces isomorphisms on residue fields.
\end{remark}

\begin{lemma}
    \label{lem:integral-epics}
    Let $R \subset S$ be an epimorphic integral extension of rings, and $\frp \in \Spec(R)$. Then, $R_\frp \to S_\frp$ is a local integral extension such that $\frp S_\frp$ is the maximal ideal of $S_\frp$.
\end{lemma}

\begin{proof}
    It is clear that $R_\frp \to S_\frp$ is an integral extension. Then $S_\frp \neq 0$ and every maximal ideal of $S_\frp$ contracts to $\frp R_\frp$ by integrality. Thus, it is enough to show that $\frp S_\frp$ is a maximal ideal of $S_\frp$. To this end, observe that we have an integral epimorphism
    \[
    \kappa(\frp) \to \kappa(\frp) \otimes_{R_\frp} S_\frp = S_{\frp}/\frp S_{\frp} \neq 0.
    \] 
    %is an epimorphism by base change. Since $S_\frp/\frp S_\frp = \kappa(\frp) \otimes_{R_\frp} S_\frp \neq 0$ by lying over, we get that $\kappa(\frp) \to S_\frp/\frp S_\frp$ 
    It is an isomorphism by \autoref{cor:epimorphisms-from-fields}, and hence $S_\frp/\frp S_\frp$ is a field.
\end{proof}

\begin{lemma}
    \label{lem:epimorphism-extension-reduction}
    Let $\varphi \colon R \to S$ be a ring map. Then $\varphi$ is an epimorphism if and only if the ring extension $\im(\varphi) \hookrightarrow S$ is an epimorphism.
\end{lemma}

\begin{proof}
    The lemma follows from the fact that compositions of epimorphisms are epimorphisms, and if $A \to B \to C$ are ring maps such that $A \to C$ is an epimorphism, then $B \to C$ is an epimorphism.
\end{proof}

\begin{corollary}
    \label{cor:integral-epics-to-dim0-Noetherian}
    Let $\varphi \colon R \to S$ be an integral epimorphism of rings such that $S$ is an Artinian ring. Then $\varphi$ is surjective.
\end{corollary}

\begin{proof}
    Replacing %$R$ by $\im(\varphi)$ and 
 $\varphi$ by the inclusion map $\im(\varphi) \subset S$ (which is an epimorphic integral extension), we may assume that $\varphi$ is an extension and show that it is an equality. %it suffices to show that if $R \subset S$ is an epimorphic integral extension and $S$ is Noetherian of dimension $0$, then $R = S$. 
    %Equivalently, for all $\frp \in \Spec(R)$, we must show that $R_\frp \hookrightarrow S_\frp$ is surjective. Thus, using 
    Then, by \autoref{lem:integral-epics}, we may further %replace $R$ by $R_\frp$ and $S$ by $S_\frp$ to 
    assume that $(S,\mathfrak{n})$ and $(R,\m)$ are local rings (with isomorphic residue fields). %$(R,\m) \hookrightarrow (S,\mathfrak{n})$ is a local integral epimorphic extension  of local rings where $S$ is Noetherian and dimension $0$. 
    %Note $R/\m \to S/\m S$ is an epimorphism by base change, and hence, an isomorphism by \autoref{cor:epimorphisms-from-fields}. 
    Since $S$ is Artinian, $\mathfrak{n}$ is nilpotent. Thus, $\m$ is nilpotent as well since $\m^n \subset \mathfrak{n}^n$ holds for all $n \geq 0$ (note $\m \subset \mathfrak{n}$ since $R \subset S$ is a local extension). Since $R/\m \to S/\m S$ is an isomorphism and $\m$ is a nilpotent ideal, $R \hookrightarrow S$ is surjective by \cite[\href{https://stacks.math.columbia.edu/tag/00DV}{Tag~00DV~(11)}]{stacks-project}.
\end{proof}

The property of being an epimorphism is local on the base.

\begin{lemma}
    \label{lem:epimorphisms-local-property}
    A ring map $\varphi \colon R \to S$ is an epimorphism if and only if for all prime ideals $\frp$ of $R$, $\varphi_\frp \colon R_\frp \to S_\frp$ is an epimorphism.
\end{lemma}

\begin{proof}
    We apply \autoref{lem:epimorphism-alternative-characterizations}~\autoref{lem:epimorphism-alternative-characterizations.d}. Localization commutes with cokernels. Now note that $S \otimes_R \coker(\varphi) = 0$ if and only if  
    \[
    (S \otimes_R \coker(\varphi)) \otimes_R R_\frp \cong S_\frp \otimes_{R_\frp} \coker(\varphi_\frp) = 0
    \]
    for all $\frp \in \Spec(R)$. This last condition is precisely that $\varphi_\frp$ is an epimorphism.
\end{proof}

The property of being an epimorphism descends under universally injective base change.

\begin{lemma}
\label{lem:epimorphism-descent}
    Let $R$ be a ring and $S_1 \to S_2$ be a map of $R$-algebras. Let $R \to T$ be universally injective. If the induced map $T \otimes_R S_1 \to T \otimes_R S_2$ is an epimorphism, then $S_1 \to S_2$ is an epimorphism.
\end{lemma}

\begin{proof}
   The result follows from the following observation: if $f, g \colon S_2 \to S'$ are two ring maps such that $\id_T \otimes_R f = \id_T \otimes_R g$, then $f = g$. Indeed, this last assertion follows from the commutative diagram,
   \[
   \begin{tikzcd}[column sep=huge]
	S_2 & {S'} \\
	{T\otimes_RS_2} & {T\otimes_RS'}
	\arrow["\tau", from=1-1, to=1-2]
	\arrow[hook, from=1-1, to=2-1]
	\arrow[hook, from=1-2, to=2-2]
	\arrow["{\id_T\otimes_R\tau}"', from=2-1, to=2-2],
\end{tikzcd}
\]
  where $\tau = f$ or $\tau = g$, because $S' \to T \otimes_R S'$ is injective.
\end{proof}

\autoref{lem:O-epimorphism} provides a situation in which an epimorphism of rings is surjective. We will recall some additional situations in which ring epimorphisms are surjective.

\begin{proposition}
    \label{prop:epic-implies-surjective}
    Let $\varphi \colon R \to S$ be an epimorphism of rings. Suppose $S \neq 0$.
    \begin{enumerate}
        \item If $S$ has the property that every non-zero quotient $R$-module of $S$ admits a non-zero quotient that is a cyclic $R$-module,\footnote{In the terminology of \cite[Expos\'e~3]{SamuelSeminar}, $S$ satisfies property $(C)$.} then $\varphi$ is surjective.\label{prop:epic-implies-surjective.a}
        \item If $\varphi$ is finite, then $\varphi$ is surjective.\label{prop:epic-implies-surjective.b}
        \item If $\varphi$ is integral and $\im(\varphi)$ is Noetherian, then $\varphi$ is surjective. \label{prop:epic-implies-surjective.c}
    \end{enumerate}
\end{proposition}

\begin{proof}
    We include the argument for \autoref{prop:epic-implies-surjective.a}, which is shown in \cite[Expos\'e~3, Proposition~7]{SamuelSeminar}. If $\coker(\varphi) \neq 0$, then $S \otimes_R \coker(\varphi) \neq 0$ follows upon choosing a non-zero cyclic $R$-module $R/I$ and a surjection $\coker(\varphi) \twoheadrightarrow R/I$. This is because then $S \otimes_R \coker(\varphi)$ surjects onto $R/I \otimes_R R/I \neq 0$. However, $S \otimes_R \coker(\varphi) \neq 0$ contradicts $\varphi$ being an epimorphism by \autoref{lem:epimorphism-alternative-characterizations}.
    
        If $S$ is a finite $R$-module, then we are in the situation \autoref{prop:epic-implies-surjective.a}. Hence, \autoref{prop:epic-implies-surjective.b} follows from \autoref{prop:epic-implies-surjective.a}.
    
       Finally, \autoref{prop:epic-implies-surjective.c} was originally proved in \cite[Expos\'e~7, Proposition~3.8~(i)]{SamuelSeminar}. We can provide a different argument because of the direct summand theorem \cite{HochsterCyclicPurity, AndreDirectsummandconjecture}. Using \autoref{lem:epimorphism-extension-reduction}, it is enough to show that if $R \subset S$ is an integral extension that is an epimorphism and $R$ is Noetherian, then $R = S$. Now, as a consequence of Ohi's equivalent reformulation of the direct summand theorem \cite{OhiDirectSummand}, we know that an integral extension with a Noetherian base descends flatness and hence satisfies property (O); see \autoref{rem:ring-maps-satisfying-O}~\autoref{rem:ring-maps-satisfying-O.b}. Thus, $R = S$ follows from \autoref{lem:O-epimorphism}.
\end{proof}

% \begin{corollary}
%     \label{cor:integral-epimorphism-from-Noetherian}
%    Let $\varphi \colon R \to S$ be an integral ring map such that $\im(\varphi)$ is Noetherian. If $\varphi$ is an epimorphism, then $\varphi$ is surjective.
% \end{corollary}

% \begin{proof}
%     $\im(\varphi) \hookrightarrow S$ is an integral extension and an epimorphism by \autoref{lem:epimorphism-extension-reduction}. Then $\im(\varphi) \hookrightarrow S$ is an isomorphism by \autoref{prop:epic-implies-surjective}~\autoref{prop:epic-implies-surjective.c} since $\im(\varphi)$ is Noetherian and integral ring maps are universally closed on $\Spec$.
% \end{proof}

\begin{remark}
    \autoref{prop:epic-implies-surjective}~\autoref{prop:epic-implies-surjective.c} fails without the Noetherian hypothesis. In \cite[Expos\'e~8]{SamuelSeminar}, Lazard constructs, for a field $k$, an epimorphism of local $k$-algebras $\varphi \colon (C,\m) \to (D,\eta)$ of dimension $0$ whose residue fields are isomorphic to $k$ and such that $\varphi$ is not surjective. Note that $\varphi$ is automatically integral. Indeed, given $x \in D$, write $x = \lambda + y$, for $\lambda \in k$ and $y \in \eta$. Then $\lambda$ is integral over $C$ because $k \subset C$ and $y$ is integral over $C$ since it is a nilpotent element. Thus, $x$ is integral over $C$ as well. In \autoref{sec:epimorphic-Frobenius-not-surjective}, we show that even purely inseparable ring epimorphisms may fail to be surjective.
\end{remark}

Finally, we recall the following ``geometric'' characterization of ring epimorphisms.

\begin{proposition}\cite[Expos\'e 4,~Proposition~1.5]{SamuelSeminar}
\label{prop:necessary-sufficient-characterization-epimorphism}
A ring map $\varphi\:R \to S$ is an epimorphism if and only if it satisfies the following conditions:
\begin{enumerate}[(i)]
    \item $\Spec(S) \to \Spec(R)$ is injective.\label{prop:necessary-sufficient-characterization-epimorphism.i}
    \item For all $\frq \in \Spec(S)$, if $\frp = R \cap \frq$ then $\kappa(\frp) \hookrightarrow \kappa(\frq)$ is %an isomorphism.
    purely inseparable.\label{prop:necessary-sufficient-characterization-epimorphism.ii}
    \item %$\ker(\mu_{\varphi})$ 
    $\ker(\mu_{\varphi})$ is a finitely generated ideal of $S \otimes_R S$.\label{prop:necessary-sufficient-characterization-epimorphism.iii}
    \item $R \to S$ is formally unramified; that is, $\ker(\mu_{\varphi}) = \ker(\mu_{\varphi})^2$.\label{prop:necessary-sufficient-characterization-epimorphism.iv} %(see \autoref{eqn.Definition_Diagonal_Ideal}). %$\ker(\mu_{\varphi}) = \ker(\mu_{\varphi})^2$.
\end{enumerate}
\end{proposition}

\begin{proof}
    Suppose $\varphi$ is an epimorphism. Then \autoref{prop:necessary-sufficient-characterization-epimorphism.i} and \autoref{prop:necessary-sufficient-characterization-epimorphism.ii} hold by \autoref{rem:injective-on-spec}, and \autoref{prop:necessary-sufficient-characterization-epimorphism.iii} and \autoref{prop:necessary-sufficient-characterization-epimorphism.iv} hold by \autoref{lem:epimorphism-alternative-characterizations} \autoref{lem:epimorphism-alternative-characterizations.e}. Conversely, suppose conditions \autoref{prop:necessary-sufficient-characterization-epimorphism.i}-\autoref{prop:necessary-sufficient-characterization-epimorphism.iv} hold. Conditions \autoref{prop:necessary-sufficient-characterization-epimorphism.i} and \autoref{prop:necessary-sufficient-characterization-epimorphism.ii} imply $\operatorname{Spec}(S) \to \operatorname{Spec}(R)$ is radicial \cite[\href{https://stacks.math.columbia.edu/tag/01S3}{Tag 01S3}]{stacks-project}, so $\mu_{\varphi} : S \otimes _R S \to S$ induces a surjection on spectra, see \cite[\href{https://stacks.math.columbia.edu/tag/01S4}{Tag 01S4}]{stacks-project}. Thus $\ker (\mu_{\varphi})$ is a locally nilpotent ideal. By \autoref{prop:necessary-sufficient-characterization-epimorphism.iii}, it follows that $\ker (\mu_{\varphi})$ is actually nilpotent, and then combining with \autoref{prop:necessary-sufficient-characterization-epimorphism.iv}, we see that $\ker (\mu_{\varphi}) = 0$, hence $\varphi$ is an epimorphism by \autoref{lem:epimorphism-alternative-characterizations} \autoref{lem:epimorphism-alternative-characterizations.e}. 
\end{proof}

\begin{remark}
\autoref{prop:necessary-sufficient-characterization-epimorphism}, as stated, is slightly different than \cite[Expos\'e 4,~Proposition~1.5]{SamuelSeminar} because the latter assumes in \autoref{prop:necessary-sufficient-characterization-epimorphism.ii} that for all $\frq \in \Spec(S)$, if $\frp = R \cap \frq$, then $\kappa(\frp) \hookrightarrow \kappa(\frq)$ is an isomorphism. Note that if $\varphi$ is an epimorphism, then all extensions of residue fields will be isomorphisms by \autoref{rem:injective-on-spec}.
\end{remark}

\section{Characterizing B-Nil Formally Unramified Maps}
\label{sec:main}

Let $p> 0$ be a prime number. A map of $R \to S$ of $\mathbf{F}_p$-algebras is \emph{b-nil formally unramified} (resp. \emph{b-nil formally smooth}) if for every solid commutative square in the category of $\mathbf{F}_p$-algebras where $I \subset A$ is an ideal with 
\[
I^{[p]} \coloneqq \langle a^p \mid a \in A \rangle = 0,
\]
there is at most one (resp. at least one) dashed arrow $S \to A$ making the diagram 
\begin{equation}
    \label{equn-thediagram}
\begin{tikzcd}
	{A/I} & S \\
	A & R
	\arrow[from=1-2, to=1-1]
	\arrow[dashed, from=1-2, to=2-1]
	\arrow[from=2-1, to=1-1]
	\arrow[from=2-2, to=1-2]
	\arrow[from=2-2, to=2-1]
\end{tikzcd}
\end{equation}
commute. We say $R \to S$ is \emph{b-nil formally \'etale} if it is b-nil formally unramified and b-nil formally smooth.

\begin{example}
\label{example:episareunramified}
    If $R \to S$ is an epimorphism of $\mathbf{F}_p$-algebras, then $R \to S$ is b-nil formally unramified. 
\end{example}

%If $\varphi: R \to S$ is a ring map, write $\mu = \mu_{S/R} = \mu_\varphi : S \otimes _R S \to S$ for the multiplication map. It is surjective with kernel the ideal generated by elements $s \otimes 1 - 1 \otimes s$. 

\begin{lemma}
\label{lem:basic-properties-b-nil-formally-unramified}
Let $\varphi \:R \to S$, $\psi\:S \to T$ be maps of $\mathbf{F}_p$-algebras.
    \begin{enumerate}
        \item If $R \to T$ is b-nil formally unramified, then $S \to T$ is b-nil formally unramified.\label{lem:basic-properties-b-nil-formally-unramified.a}
        \item Let $A$ be an $\mathbf{F}_p$-algebra and $I$ be an ideal. Then $A \twoheadrightarrow A/I$ is b-nil formally \'etale if and only if $I = I^{[p]}$.\label{lem:basic-properties-b-nil-formally-unramified.b}
        \item \label{lem:basic-properties-b-nil-formally-unramified.c} The following are equivalent:
        \begin{enumerate}[(i)]
            \item $\varphi$ is b-nil formally unramified.\label{lem:basic-properties-b-nil-formally-unramified.c.i}
            \item 
            %$\ker(\mu_{S/R})^{[p]} = \ker(\mu_{S/R})$.
            $\ker(\mu_{\varphi}) = \ker(\mu_{\varphi})^{[p]}$.
            \label{lem:basic-properties-b-nil-formally-unramified.c.ii}
            \item $\mu_{\varphi}$ is b-nil formally \'etale.\label{lem:basic-properties-b-nil-formally-unramified.c.iii}
            \item $\mu_{\varphi}$ is b-nil formally smooth. \label{lem:basic-properties-b-nil-formally-unramified.c.iv}
        \end{enumerate}
    \end{enumerate}
\end{lemma}

\begin{proof}
    \autoref{lem:basic-properties-b-nil-formally-unramified.a} is standard. \autoref{lem:basic-properties-b-nil-formally-unramified.b} follows from \cite[Theorem~3.13]{DattaOlanderRelFrobIso} because the relative Frobenius of $A \to A/I$ can be identified with the canonical surjection $A/I^{[p]} \twoheadrightarrow A/I$, and this last map is an isomorphism precisely when $I = I^{[p]}$.

    Finally, for \autoref{lem:basic-properties-b-nil-formally-unramified.c},
    we have \autoref{lem:basic-properties-b-nil-formally-unramified.c.i}
   $\Longleftrightarrow$ \autoref{lem:basic-properties-b-nil-formally-unramified.c.iv} because giving a pair of dashed arrows that make the diagram \autoref{equn-thediagram} commute is equivalent to providing a solid commutative square 
\[\begin{tikzcd}
	{A/I} & S \\
	A & {S\otimes _R S}
	\arrow[from=1-2, to=1-1]
	\arrow["\beta", dashed, from=1-2, to=2-1]
	\arrow[from=2-1, to=1-1]
	\arrow["{\mu_{S/R}}"', from=2-2, to=1-2]
	\arrow[from=2-2, to=2-1]
\end{tikzcd}\]
    by the universal property of the tensor product, and the two dashed arrows are equal if and only if there exists a dashed arrow $\beta$ that makes the corresponding diagram above commute. 

    Now clearly \autoref{lem:basic-properties-b-nil-formally-unramified.c.iii} implies \autoref{lem:basic-properties-b-nil-formally-unramified.c.iv}. If the equivalent conditions \autoref{lem:basic-properties-b-nil-formally-unramified.c.i} and \autoref{lem:basic-properties-b-nil-formally-unramified.c.iv} hold, then so does \autoref{lem:basic-properties-b-nil-formally-unramified.c.iii}: It suffices to show that $\mu_{\varphi}$ is b-nil formally unramified, and this is true by \autoref{lem:basic-properties-b-nil-formally-unramified.a} since the b-nil formally unramified map $R \to S$ factors through $\mu_{\varphi}$. Finally, the equivalence of \autoref{lem:basic-properties-b-nil-formally-unramified.c.ii} and \autoref{lem:basic-properties-b-nil-formally-unramified.c.iii} holds by \autoref{lem:basic-properties-b-nil-formally-unramified.b} because $\mu_{\varphi}$ is surjective. 
\end{proof}

%To give our characterization of b-nil formal unramification in positive characteristic, we recall the \emph{(relative) Frobenius morphism} attached to any homomorphism $\varphi\:R \to S$ of $\bF_p$-algebras. This is the morphism $F_{\varphi}$ defined by the Cartesian diagram
%\[
%\begin{tikzcd}
%	R & S & \\
%	{R} & {S^{(p)}} \\
%	&& {S}
%	\arrow["{\varphi}",from=1-1, to=1-2]
%	\arrow["{F_R}"', from=1-1, to=2-1]
%	\arrow[from=1-2, to=2-2]
%	\arrow["{F_S}", curve={height=-18pt}, from=1-2, to=3-3]
%	\arrow[from=2-1, to=2-2]
%	\arrow["{\varphi}",curve={height=18pt}, from=2-1, to=3-3]
%	\arrow["{F_{\varphi}}", from=2-2, to=3-3]
%    \end{tikzcd}
%    \]
%In other words, $F_{\varphi} (r \otimes s) = \varphi(r)s^p$. We denote the image of $F_{\varphi}$ as $R[S^p] \subset S$. 

Let $\varphi \colon R \to S$ be a homomorphism of $\mathbf{F}_p$-algebras and $F_{\varphi}$ the associated relative Frobenius. The cokernel of $F_{\varphi}$ is often denoted by $\sB^1_{\varphi}$ in the literature because of its relationship with Kähler differentials (and the de Rham complex) through the Cartier isomorphisms. This is, of course, an $S^{(p)}$-module. To make it an $S$-module, we pull it back along $F_{\varphi}$ to get
\[
 \mathfrak{B}_{\varphi} \coloneqq S \otimes_{S^{(p)}} \coker(F_\varphi)  \cong \ker(\mu_{F_{\varphi}}) \cong \ker(\mu_{\varphi})/\ker(\mu_{\varphi})^{[p]}. 
\]
As we will see shortly, this module plays a role in b-nil formal unramification that is analogous to the role played by $\Omega_{\varphi}$ in the case of formal unramification. Thus, it is a sort of Frobenius-theoretic cotangent module.

\begin{lemma}
\label{lem:diagonal-relative-Frobenius}
Let $\varphi \colon R \to S$ be a map of $\mathbf{F}_p$-algebras. % and $F_{\varphi} : S^{(p)} \to S$ be its Frobenius. 
Then
\[
\ker(\mu_{F_{\varphi}})^{[p]} = 0.
\]
%Thus $F_{\varphi} (r \otimes s) = \varphi(r)s^p$. Write $J$ for the kernel of $\mu_{F_\varphi}$. Then $J^{[p]} = 0$.
\end{lemma}

\begin{proof}
    Note that $\ker(\mu_{F_{\varphi}}) \subset S \otimes _{S^{(p)}} S$ is generated by elements $s \otimes 1 - 1 \otimes s$ with $s \in S$. It suffices to show that the $p^{th}$ power of any such element is zero. But $s^p \otimes 1 = 1 \otimes s^p$ in $S \otimes _{S^{(p)}} S$ as $s^p = F_{\varphi}(s \otimes 1)$. 
\end{proof}

\begin{theorem}
\label{thm:b-nil-formally-unramified}
   % Suppose $\varphi \colon R \to S$ is a map of $\mathbf{F}_p$-algebras. Let $F_\varphi \colon S^{(p)} \to S$ be the relative Frobenius. 
   Let $\varphi \colon R \to S$ be a map of $\mathbf{F}_p$-algebras. Then the following are equivalent: 
    \begin{enumerate}
        \item $\varphi$ is b-nil formally unramified.\label{thm:b-nil-formally-unramified.a}
        \item $F_\varphi$ is b-nil formally unramified.\label{thm:b-nil-formally-unramified.b}
        \item $F_{\varphi}$ is an epimorphism of rings.\label{thm:b-nil-formally-unramified.c}
        \item %$S \otimes_{S^{(p)}} \coker(F_\varphi) = 0$. 
        $\mathfrak{B}_{\varphi} = 0$. \label{thm:b-nil-formally-unramified.d}
    \end{enumerate}
\end{theorem}

\begin{proof}
    \autoref{thm:b-nil-formally-unramified.a} $\implies$ \autoref{thm:b-nil-formally-unramified.b}: By \autoref{lem:basic-properties-b-nil-formally-unramified}~\autoref{lem:basic-properties-b-nil-formally-unramified.a}, since the ring map $\varphi$ factors through $F_{\varphi}$, it follows that $F_{\varphi}$ is b-nil formally unramified. 
    
    \autoref{thm:b-nil-formally-unramified.b} $\implies$ \autoref{thm:b-nil-formally-unramified.c}: By \autoref{lem:epimorphism-alternative-characterizations}, it suffices to show that $\mu_{F_\varphi}$ is an isomorphism.  We have  $\mu_{F_{\varphi}}$ is b-nil formally \'etale by \autoref{lem:basic-properties-b-nil-formally-unramified}~\autoref{lem:basic-properties-b-nil-formally-unramified.c}. Then \[
\ker(\mu_{F_{\varphi}}) = \ker(\mu_{F_{\varphi}})^{[p]} = 0
    \]by \autoref{lem:basic-properties-b-nil-formally-unramified}~\autoref{lem:basic-properties-b-nil-formally-unramified.c} and \autoref{lem:diagonal-relative-Frobenius}, i.e., $\mu_{F_{\varphi}}$ is an isomorphism.

   \autoref{thm:b-nil-formally-unramified.c} $\implies$ \autoref{thm:b-nil-formally-unramified.a} is well-known. See for instance, \cite[Lemma~3.14~(b)]{DattaOlanderRelFrobIso}. 

   Finally, \autoref{thm:b-nil-formally-unramified.c}$\Longleftrightarrow$\autoref{thm:b-nil-formally-unramified.d} by \autoref{lem:epimorphism-alternative-characterizations}.
\end{proof}

%\begin{corollary} --- this is immediate from the definition, I added it as an example following the definition but we can remove it if we like. 
%    \label{cor:surjections-bnil-formally-unramified}
  %  A surjection of $\mathbf{F}_p$-algebras is b-nil formally unramified.
%\end{corollary}
%
%\begin{proof}
%    Indeed, the relative Frobenius of a surjective map of $\mathbf{F}_p$-algebras is surjective (since the relative Frobenius factors the surjection), and hence it is an epimorphism.
%\end{proof}

\begin{corollary}
    \label{corformally-unramified-fields}
    An extension of fields of characteristic $p > 0$ is formally unramified if and only if it is b-nil formally unramified.
\end{corollary}

\begin{proof}
    Let $L/K$ be a formally unramified extension of fields of characteristic $p$. Since a $p$-basis of $L/K$ yields a basis of $\Omega_{L/K}$  (\cite[\href{https://stacks.math.columbia.edu/tag/07P2}{Tag 07P2}]{stacks-project}), we see that $\emptyset$ is a $p$-basis of $L/K$. Then $K[L^p] = L$, that is, $F_{L/K}$ is surjective. Thus, $L/K$ is b-nil formally unramified by \autoref{thm:b-nil-formally-unramified}.
\end{proof}

\begin{remark}
    One can use \autoref{thm:b-nil-formally-unramified} to give a slightly different proof of \cite[Theorem~3.13]{DattaOlanderRelFrobIso}, where it is shown that if $R \to S$ is a b-nil formally \'etale map of $\mathbf{F}_p$-algebras, then $F_{S/R}$ is an isomorphism. Namely, we have $F_\varphi$ is an epimorphism by b-nil formal unramifiedness, and $F_\varphi$ is flat (in fact, $S$ is a projective $S^{(p)}$-module) by b-nil formal smoothness \cite[Proposition~3.9]{DattaOlanderRelFrobIso}. Then $F_\varphi$ is faithfully flat as $\Spec(F_\varphi)$ is a homeomorphism. But a faithfully flat epimorphism is an isomorphism by \autoref{cor:universally-injective-epimorphism}.
\end{remark}

\begin{corollary}
    \label{cor:b-nil-formally-unramified-over-epi-Frobenius}
    Let $k$ be an $\mathbf{F}_p$-algebra such that its (absolute) Frobenius $F_k$ is an epimorphism (for example, if $k$ is semi-perfect). Let $R$ be a $k$-algebra. The following are equivalent:
    \begin{enumerate}
        \item $k \to R$ is b-nil formally unramified.\label{cor:b-nil-formally-unramified-over-epi-Frobenius.a}
        \item $F_R$ is an epimorphism.\label{cor:b-nil-formally-unramified-over-epi-Frobenius.b}
        \item $F_R$ is b-nil formally unramified.\label{cor:b-nil-formally-unramified-over-epi-Frobenius.c}
        \item $\mathbf{F}_p \to R$ is b-nil formally unramified.\label{cor:b-nil-formally-unramified-over-epi-Frobenius.d}
    \end{enumerate}
\end{corollary}

\begin{proof}
    Consider the diagram that characterizes the relative Frobenius $F_{R/k}$:
    \[\begin{tikzcd}
	k & R & \\
	{k} & {R^{(p)}} \\
	&& {R.}
	\arrow[from=1-1, to=1-2]
	\arrow["{F_k}"', from=1-1, to=2-1]
	\arrow[from=1-2, to=2-2]
	\arrow["{F_R}", curve={height=-18pt}, from=1-2, to=3-3]
	\arrow[from=2-1, to=2-2]
	\arrow[curve={height=18pt}, from=2-1, to=3-3]
	\arrow["{F_{R/k}}", from=2-2, to=3-3]
\end{tikzcd}\]
Since epimorphisms are preserved under base change (\autoref{cor:epimorphisms-formally-unramified}), $R \to R^{(p)}$ is an epimorphism.

\autoref{cor:b-nil-formally-unramified-over-epi-Frobenius.a}$\implies$\autoref{cor:b-nil-formally-unramified-over-epi-Frobenius.b}: If $k \to R$ is b-nil formally unramified, then $F_{R/k}$ is an epimorphism by \autoref{thm:b-nil-formally-unramified}. Hence, $F_R$ is an epimorphism by the above diagram since it is a composition of two epimorphisms.

\autoref{cor:b-nil-formally-unramified-over-epi-Frobenius.b}$\implies$\autoref{cor:b-nil-formally-unramified-over-epi-Frobenius.c}: Epimorphisms are b-nil formally unramified by \autoref{example:episareunramified}.
%$F_R$ can be identified with $F_{R/\mathbf{F}_p}$, the relative Frobenius of $\mathbf{F}_p \to R$. Thus, by hypothesis and \autoref{thm:b-nil-formally-unramified}, $F_{R/\mathbf{F}_p}$ is an epimorphism and hence b-nil formally unramified. Then $F_R$ is also b-nil formally unramified.

\autoref{cor:b-nil-formally-unramified-over-epi-Frobenius.c}$\implies$\autoref{cor:b-nil-formally-unramified-over-epi-Frobenius.d}: %As in the previous implication, t
This follows upon identifying $F_R$ with $F_{R/\mathbf{F}_p}$ and applying \autoref{thm:b-nil-formally-unramified}.

\autoref{cor:b-nil-formally-unramified-over-epi-Frobenius.d}$\implies$\autoref{cor:b-nil-formally-unramified-over-epi-Frobenius.a}: $\mathbf{F}_p \to R$ factors as $\mathbf{F}_p \to k \to R$. Thus, $\mathbf{F}_p \to k \to R$ is b-nil formally unramified, and so, $k \to R$ is b-nil formally unramified by \autoref{lem:basic-properties-b-nil-formally-unramified}~\autoref{lem:basic-properties-b-nil-formally-unramified.a}.
\end{proof}

The b-nil formally unramified property satisfies universally injective descent (the universally injective descent of the formally unramified property is well-known \cite[\href{https://stacks.math.columbia.edu/tag/08XE}{Tag~08XE}]{stacks-project}).

\begin{corollary}
    \label{cor:bnil-formally-unramified-descent}
    Let $R \to S$ be a b-nil formally unramified map of $\mathbf{F}_p$-algebras. Let $R \to T$ be a universally injective ring map. If the induced map $T \to T \otimes_R S$ is b-nil formally unramified, then $R \to S$ is b-nil formally unramified.
\end{corollary}

\begin{proof}
    The base change of the $R$-algebra map $F_{S/R}$ along $R \to T$ yields $F_{T \otimes_R S/T}$ since the formation of relative Frobenius commutes with base change \cite[Expos\'e~XV,~$n^02$,~Proposition~1~b)]{SGA5}. Thus, if $F_{T \otimes_R S/S}$ is an epimorphism, then $F_{S/R}$ is an epimorphism by \autoref{lem:epimorphism-descent}. The result now follows by \autoref{thm:b-nil-formally-unramified}.
\end{proof}

\begin{definition}
    \label{def:F-pure-homomorphisms}
    A map of $\mathbf{F}_p$-algebras $\varphi \colon R \to S$ is \emph{$F$-pure} if $F_\varphi$ is universally injective (aka a pure ring map).
\end{definition}

The above notion, first explored by \cite{HashimotoFpure}, is the relative version of the notion of \emph{$F$-purity} of an $\mathbf{F}_p$-algebra $R$. Recall, the latter is the condition that the absolute Frobenius $F_R$ is universally injective. Thus, $R$ is $F$-pure (in the absolute sense) if and only if $\mathbf{F}_p \to R$ is an $F$-pure homomorphism (in the relative sense). 

For $F$-pure homomorphisms, \autoref{thm:b-nil-formally-unramified} implies:

\begin{corollary}
    \label{cor:bnil-unramified-etale-equal-F-pure}
    Let $\varphi \colon R \to S$ be an $F$-pure map of $\mathbf{F}_p$-algebras. Then $\varphi$ is b-nil formally unramified if and only if $\varphi$ is b-nil formally \'etale.
\end{corollary}

\begin{proof}
    The non-trivial implication is $\implies$. If $\varphi$ is b-nil formally unramified, then $F_\varphi$ is a universally injective epimorphism by \autoref{thm:b-nil-formally-unramified} and our hypothesis. A universally injective epimorphism is an isomorphism by \autoref{cor:universally-injective-epimorphism}. Thus, $F_\varphi$ is an isomorphism, and so, $\varphi$ is b-nil formally \'etale by \cite[Theorem~3.21]{DattaOlanderRelFrobIso}.
\end{proof}

\begin{remark}
    The full strength of universal injectivity of $F_\varphi$ is not needed in \autoref{cor:bnil-unramified-etale-equal-F-pure}. One can instead assume $F_\varphi$ is injective and satisfies the property that for every quotient module $M$ of the $S^{(p)}$-module $S$, if $S \otimes_{S^{(p)}} M = 0$, then $M = 0$. Here $S^{(p)}$ is the domain of $F_\varphi$. See \autoref{rem:weaker-than-O}.
\end{remark}

Our next goal is to deduce an analog for b-nil formal unramifiedness (see \autoref{thm:F-finite-b-nil-formally-unramified}) of \cite[Theorem~3.20]{DattaOlanderRelFrobIso} for b-nil formal smoothness and \cite[Theorem~4.1]{DattaOlanderRelFrobIso} for b-nil formal \'etaleness.
For this, recall the following relative version of $F$-finiteness due to Hashimoto \cite{HashimotoFfinite}.

\begin{definition}
    \label{def:F-finite-ring-map}
    A map of $\mathbf{F}_p$-algebras $R \to S$ is \emph{$F$-finite} if $F_{S/R}$ is a finite (equivalently, finite type) map.
\end{definition}

\begin{example}
    If $S$ is $F$-finite, then any map of $\mathbf{F}_p$-algebras $R \to S$ is $F$-finite because the absolute Frobenius of $S$ factors via the relative Frobenius $F_{S/R}$.
\end{example}

%\begin{lemma}
%    \label{lem:purely-inseparable-unramified}
%    Let $R \hookrightarrow S$ be a formally unramified extension of rings of prime characteristic $p > 0$ such that $S^p \subset R$. For all $\frq \in \Spec(S)$, if $\frp = \frq \cap R$, then $\kappa(\frp) \hookrightarrow \kappa(\frq)$ is an isomorphism.
%\end{lemma}
%
%\begin{proof}
%    Since $R \hookrightarrow S$ is formally unramified, so is $\kappa(\frp) \hookrightarrow \kappa(\frq)$, and hence, $\emptyset$ is a $p$-basis of $\kappa(\frq)/\kappa(\frp)$. By the assumption that $S^p \subset R$, we have $\kappa(\frq)^p \subset \kappa(\frp)$. Hence, $\kappa(\frp) = \kappa(\frp)[\kappa(\frq)^p] = \kappa(\frq)$, where the last equality follows from the empty $p$-basis observation.
%\end{proof}

\begin{theorem}
\label{thm:F-finite-b-nil-formally-unramified}
    Let $\varphi \colon R \to S$ be an $F$-finite map of $\mathbf{F}_p$-algebras (for example, if $S$ is $F$-finite). Then the following are equivalent:
    \begin{enumerate}
        \item \label{thm:F-finite-b-nil-formally-unramified.a} $\varphi$ is b-nil formally unramified.
        \item \label{thm:F-finite-b-nil-formally-unramified.b} $F_\varphi$ is surjective.
        %\item \label{thm:F-finite-b-nil-formally-unramified.c} $F_\varphi$ is an epimorphism.
        \item \label{thm:F-finite-b-nil-formally-unramified.c} $\varphi$ is formally unramified.
    \end{enumerate}
\end{theorem}

\begin{proof}
    We have \autoref{thm:F-finite-b-nil-formally-unramified.b} $\implies$ \autoref{thm:F-finite-b-nil-formally-unramified.a} by \autoref{thm:b-nil-formally-unramified} and  \autoref{thm:F-finite-b-nil-formally-unramified.a} $\implies$ \autoref{thm:F-finite-b-nil-formally-unramified.c} is clear. By our assumption, $F_\varphi$ is finite. A finite epimorphism is a surjection by \autoref{prop:epic-implies-surjective}. Thus, to show \autoref{thm:F-finite-b-nil-formally-unramified.c} $\implies$ \autoref{thm:F-finite-b-nil-formally-unramified.b} it is enough to show $F_\varphi$ is an epimorphism if $\varphi$ is formally unramified. Note that $\varphi$ being formally unramified implies $F_\varphi$ is also formally unramified because $\varphi = F_\varphi \circ (\id_{F_*R} \otimes_R \varphi)$. 
    Let $R[S^p]$ be the image of $F_\varphi$. Then $R[S^p] \hookrightarrow S$ is formally unramified. Note, $\Spec(S) \to \Spec(R[S^p])$ is a homeomorphism and induces purely inseparable extensions of residue fields, %, $R[S^p] \hookrightarrow S$ induces trivial extensions of residue fields by \autoref{lem:purely-inseparable-unramified}, 
    and $\ker(\mu_{S/R[S^p]})$ is a finitely generated ideal of $S \otimes_{R[S^p]} S$ because $R[S^p] \hookrightarrow S$ is finite hence of finite type. Then $R[S^p] \hookrightarrow S$ is an epimorphism by \autoref{prop:necessary-sufficient-characterization-epimorphism}, so $F_\varphi = S^{(p)} \twoheadrightarrow R[S^p] \hookrightarrow S$ is also an epimorphism.
\end{proof}

\begin{remark}
\label{rem:finiteness-hypothesis-on-relative-Frobenius}
    One obtains a robust theory of unramified maps assuming only that the map is of finite type instead of finite presentation. This is the approach of Raynaud \cite{RaynaudHenselian} and \cite[\href{https://stacks.math.columbia.edu/tag/00US}{Tag 00US}]{stacks-project}. On the other hand, the notions smooth and \'etale still need finite presentation hypothesis. This dichotomy is also reflected in \autoref{thm:F-finite-b-nil-formally-unramified}. Namely, we can show the equivalence between b-nil formally unramified and formally unramified for a map $R \to S$ of $\mathbf{F}_p$-algebras assuming $F_{S/R}$ is of finite type. However, to show the equivalence between b-nil formally smooth and formally smooth (resp. b-nil formally \'etale and formally \'etale) we need to assume $F_{S/R}$ is of finite presentation. In fact, \cite[Remark~4.2]{DattaOlanderRelFrobIso} shows that formally smooth/\'etale does not imply b-nil formally smooth/\'etale assuming $F_{S/R}$ is only of finite type. 
\end{remark}

\begin{corollary}
    \label{cor:unramified-implies-bnil}
    An unramified map of $\mathbf{F}_p$-algebras is b-nil formally unramified.
\end{corollary}

\begin{proof}
    An unramified map is of finite type and formally unramified. Thus, the relative Frobenius of an unramified map of $\mathbf{F}_p$-algebras is of finite type, that is, such a map is $F$-finite. The result now follows by \autoref{thm:F-finite-b-nil-formally-unramified}.
\end{proof}

\begin{corollary}
    \label{cor:semi-perfect}
    Let $\varphi \colon R \to S$ be an $F$-finite map of $\mathbf{F}_p$-algebras where $R$ is semi-perfect. Then the following are equivalent:
    \begin{enumerate}
        \item $\varphi$ is b-nil formally unramified.\label{cor:semi-perfect.a}
        \item $\varphi$ is formally unramified.\label{cor:semi-perfect.b}
        \item $S$ is semi-perfect.\label{cor:semi-perfect.c}
    \end{enumerate}
\end{corollary}

\begin{proof}
    We have \autoref{cor:semi-perfect.a} $\Longleftrightarrow$ \autoref{cor:semi-perfect.b} by \autoref{thm:F-finite-b-nil-formally-unramified}. If $S$ is semi-perfect, then $F_{S/R}$ is surjective; hence, $R \to S$ is b-nil formally unramified by \autoref{thm:F-finite-b-nil-formally-unramified}. Thus, \autoref{cor:semi-perfect.c} $\implies$ \autoref{cor:semi-perfect.a}.  Conversely, if $R \to S$ is b-nil formally unramified, then $F_{\varphi}$ is surjective by \autoref{thm:F-finite-b-nil-formally-unramified}, and $S \to S^{(p)}$ is surjective by base change of the surjective absolute Frobenius of $R$. Thus, the composition 
    \[
    F_S = S \to S^{(p)} \xrightarrow{F_{\varphi}} S
    \] is surjective, i.e., $S$ is semi-perfect.
\end{proof}

An examination of the proof of \autoref{thm:F-finite-b-nil-formally-unramified} shows that we have the following characterization of when the formally unramified and b-nil formally unramified properties coincide.

\begin{proposition}
    \label{prop:when-do-formally-b-nil-unramified-coincide}
    Let $\varphi \colon R \to S$ be a map of $\mathbf{F}_p$-algebras. Then the following are equivalent:
    \begin{enumerate}
        \item \label{prop:when-do-formally-b-nil-unramified-coincide.a} $\varphi$ is b-nil formally unramified.

        \item \label{prop:when-do-formally-b-nil-unramified-coincide.b} $\varphi$ is formally unramified and $\ker(\mu_{S/R[S^p]})$ is a finitely generated ideal of $S \otimes_{R[S^p]} S$.

        \item \label{prop:when-do-formally-b-nil-unramified-coincide.c} $\varphi$ is formally unramified and $\ker(\mu_{S/R[S^p]})$ is generated by an idempotent element. %of $S \otimes_{R[S^p]} S$.
    \end{enumerate}
\end{proposition}

\begin{proof}
    As the proof of \autoref{thm:F-finite-b-nil-formally-unramified} demonstrates, both \autoref{prop:when-do-formally-b-nil-unramified-coincide.a} and \autoref{prop:when-do-formally-b-nil-unramified-coincide.b} are equivalent to $F_\varphi$ being an epimorphism because of \autoref{prop:necessary-sufficient-characterization-epimorphism}. Clearly \autoref{prop:when-do-formally-b-nil-unramified-coincide.c}$\implies$\autoref{prop:when-do-formally-b-nil-unramified-coincide.b}. To see \autoref{prop:when-do-formally-b-nil-unramified-coincide.b}$\implies$\autoref{prop:when-do-formally-b-nil-unramified-coincide.c}, note that if $\varphi$ is b-nil formally unramified, then so is $F_\varphi$, and hence also $R[S^p] \hookrightarrow S$ (both are applications of \autoref{lem:basic-properties-b-nil-formally-unramified}~\autoref{lem:basic-properties-b-nil-formally-unramified.a}). In particular, $R[S^p] \hookrightarrow S$ is formally unramified; that is, $\ker(\mu_{S/R[S^p]}) = \ker(\mu_{S/R[S^p]})^2$. But a finitely generated idempotent ideal of a ring must be generated by an idempotent element, as implied by the Cayley--Hamilton theorem.
\end{proof}

\begin{remark}
    \label{rem:relative-Frobenius-monic}
    Let $R \to S$ be a map of $\mathbf{F}_p$-algebras. By \cite[Theorem~3.21]{DattaOlanderRelFrobIso}, $R \to S$ is b-nil formally \'etale if and only if $F_{S/R}$ is an isomorphism, and we show that $R \to S$ is b-nil formally unramified if and only if $F_{S/R}$ is an epimorphism. So it is natural to ask whether $F_{S/R}$ being a monomorphism (i.e., injective) characterizes b-nil formal smoothness. Alas, this is not the case. For example, for any reduced $\mathbf{F}_p$-algebra $A$, $F_{A/\mathbf{F}
    _p}$ can be identified with the absolute Frobenius of $A$, and hence it is injective. But clearly, $\mathbf{F}_p \to A$ cannot be b-nil formally smooth in general when $A$ is reduced. For example, take a Noetherian $A$ that is reduced but not regular. Then $\mathbf{F}_p \to A$ is not a regular map; hence, $\mathbf{F}_p \to A$ is not (b-nil) formally smooth by \cite[$0_{\text{IV}}$,~Corollaire~19.6.5]{EGAIV_I}.
\end{remark}

\begin{remark}
Recall that given a ring map $\varphi\colon R \to S$, the module of $n$-principal parts is given by 
\[
P^n_{\varphi} \coloneqq (S \otimes_R S)/\ker(\mu_{\varphi})^{n+1} = S \oplus \ker(\mu_{\varphi})/\ker(\mu_{\varphi})^{n+1}.
\] 
%where $I = \ker \mu_{S/R}$. 
This is an $S$-module by the action $s \cdot a \otimes b = sa \otimes b$. The universal differential operator $d\colon S \to P^n_{\varphi}$ is given by $s \mapsto 1 \otimes s$. Then there is a natural isomorphism 
\[
\Hom_S(P^n_{\varphi}, M) \xrightarrow{\psi \mapsto \psi \circ d} D^n_R(S,M)
\] 
Here, $D^n_R(S,M)$ denotes $R$-linear differential operators $S \to M$ of order $\leq n$. We will write $D_R(S,M) = \bigcup_{n \geq 0} D_R^n(S,M)$. Note that if $\varphi$ is formally unramified, then 
$$\ker (\mu_{\varphi}) = \ker (\mu_{\varphi})^2 = \ker (\mu_{\varphi})^3 = \cdots,$$
hence $P^n_{\varphi} = S$ for all $n\geq 0$ and so $D^n_R(S,M) = M$ for all $n\geq 0$ and $S$-modules $M$. Conversely, if $D^1_R(S,M) = M$ for all $S$-modules $M$, then $\varphi$ is formally unramified. One might ask whether $D_R(S,S) = S$ implies $\varphi$ is formally unramified, but this is false. In the paper \cite{MukhopadhyaySmithTrivialDiffOps}, Mukhopadhyay and Smith construct examples of $k$-algebras $A$, with $k$ an arbitrary field, such that $D_k(A,A) = A$ but $k \to A$ is not formally unramified, see \cite[Theorem 1.1 and Remark 4.6]{MukhopadhyaySmithTrivialDiffOps}.

Replacing $\ker(\mu_{\varphi})^n$ with $\ker(\mu_{\varphi})^{[p^e]}$ for $e \geq 0$ and writing 
\[
P^{[p^e]}_{\varphi} \coloneqq (S \otimes_R S)/\ker(\mu_{\varphi})^{[p^e]}  = S \oplus \ker(\mu_{\varphi})/\ker(\mu_{\varphi})^{[p^e]},
\] 
we similarly obtain a natural isomorphism 
\[
D_R^{[p^e]}(S,M) \cong \Hom_S\Big(P^{[p^e]}_{\varphi}, M\Big), 
\] 
where $D^{[p^e]}_R(S,M)$ denotes the $\ker(\mu_\varphi)^{[p^e]}$-differential operators (see \cite[Section 2.2, especially Definition 2.10]{BrennerJeffriesNunezBetancourtQuantifyingSingularities}). In particular, $P_{\varphi}^{[p]} = S \oplus \ker (\mu_{\varphi})/\ker(\mu_{\varphi})^{[p]} = \mathfrak{B}_{\varphi}$.
From this we conclude that $\varphi$ is b-nil formally unramified if and only if $\mathfrak{B}_{\varphi} = 0$, if and only if 
\[
\ker(\mu_{\varphi}) = \ker (\mu_{\varphi})^{[p]} = \ker (\mu_{\varphi})^{[p^2]} = \cdots ,\]
if and only if $P^{[p^e]}_{\varphi} =S$ for all $e \geq 0$, if and only if 
%and so
%\[
%D^{[p]}_R(S,M) \cong M \oplus \Hom_S(\mathfrak{B}_{\varphi},M).
%\]
%From this, we conclude that $\varphi$ is b-nil formally unramified if and only if 
\[
D^{[p^e]}_R(S,M) = M
\]
for all $S$-modules $M$ and $e \geq 0$. Plugging in $M=S$, we obtain that if $\varphi$ is b-nil formally unramified, then %If$\varphi$ is b-nil formally unramified, then $I^{[p^e]} = I$ by \autoref{lem:basic-properties-b-nil-formally-unramified}~\autoref{lem:basic-properties-b-nil-formally-unramified.c.ii}. 
 %In particular, 
for all $e \geq 0$, we have $D_R^{[p^e]}(S,S) \cong S$ and so there are no non-trivial $\ker(\mu_{\varphi})^{[p^e]}$-differential operators to $S$. The converse does not hold since we have $D^n_R(S,S) \subseteq D^{[p^e]}_R(S,S)$ for $n \leq p^e-1$ since $S \otimes _R S/\ker(\mu_{\varphi})^{[p^e]} \twoheadrightarrow S \otimes _R S/\ker(\mu_{\varphi})^{n}$, see \cite[Lemma 1.4.8 a)]{YekutieliExplicitConstructionResidueComplex}). Then we can use \cite[Theorem 1.1]{MukhopadhyaySmithTrivialDiffOps}. 
% This shows that if $R \to S$ is b-nil formally unramified, the ring of $R$-linear differential operators on $S$ is trivial; that is, $D_R(S,S) = S$. ---- I think it is more elementary and stronger to say that if R \to S is formally unramified, then D_R(S,S) = S, see above.
  %In \cite{MukhopadhyaySmithTrivialDiffOps}, Mukhopadhyay and Smith construct examples of $k$-algebras $R$, with $k$ an arbitrary field, that admit no non-trivial differential operators but are not formally unramified, see \cite[Theorem 1.1 and Remark 4.6]{MukhopadhyaySmithTrivialDiffOps}. Such $k \to R$ are a fortiori not b-nil formally unramified. %So, for perfect $k$, these are not b-nil formally unramified by \autoref{prop:Noetherian-over-epimorphic-Frobenius}.
\end{remark}

\section{Some Consequences for Noetherian Rings}
\label{sec:noetherian}

With additional Noetherianity assumptions, one can improve \autoref{thm:b-nil-formally-unramified} as follows.

\begin{proposition}
    \label{prop:b-nil-formally-unramified-image-relFrobenius-Noetherian}
    Let $\varphi \colon R \to S$ be a map of $\mathbf{F}_p$-algebras. Assume that $R[S^p]$, the image of $F_\varphi$ in $S$, is Noetherian. Then the following are equivalent:
    \begin{enumerate}
        \item \label{prop:b-nil-formally-unramified-image-relFrobenius-Noetherian.a} $\varphi$ is b-nil formally unramified.
        \item \label{prop:b-nil-formally-unramified-image-relFrobenius-Noetherian.b} $F_\varphi$ is an epimorphism.
        \item \label{prop:b-nil-formally-unramified-image-relFrobenius-Noetherian.c} $F_\varphi$ is surjective.
        \item \label{prop:b-nil-formally-unramified-image-relFrobenius-Noetherian.d} $\varphi$ is formally unramified and $S$ is Noetherian.
    \end{enumerate}
\end{proposition}

The proof of this proposition will use the following result.

\begin{proposition}
    \label{prop:epimorphism-surjective-Noetherian}
    Let $\psi \colon A \to B$ be a map of Noetherian $\mathbf{F}_p$-algebras. Suppose there exists a ring map $\theta \colon B \to A$ such that $\theta \circ \psi = F^e_A$ and $\psi \circ \theta = F^e_B$, for some integer $e > 0$. Then $\psi$ is formally unramified (if and) only if it is surjective.
\end{proposition}

\begin{proof}
 As stated, the result appears in \cite[Theorem~11.2]{MaPolstraFSing}, where it is credited to Andr\'e \cite[Proposition~57]{AndreHomologyFrobenius}.
\end{proof}

\begin{proof}[Proof of \autoref{prop:b-nil-formally-unramified-image-relFrobenius-Noetherian}]
    \autoref{prop:b-nil-formally-unramified-image-relFrobenius-Noetherian.a} $\Longleftrightarrow$ \autoref{prop:b-nil-formally-unramified-image-relFrobenius-Noetherian.b} by \autoref{thm:b-nil-formally-unramified}. To show \autoref{prop:b-nil-formally-unramified-image-relFrobenius-Noetherian.b}   $\Longleftrightarrow$\autoref{prop:b-nil-formally-unramified-image-relFrobenius-Noetherian.c}, it is enough to show that if $F_\varphi$ is an epimorphism then $F_\varphi$ is surjective since the other implication is trivial. Since $R[S^p]$ is the image of the integral ring map $F_\varphi$, if $F_\varphi$ is an epimorphism, then it is surjective by \autoref{prop:epic-implies-surjective}~\autoref{prop:epic-implies-surjective.c}. It remains to show \autoref{prop:b-nil-formally-unramified-image-relFrobenius-Noetherian.c}$\Longleftrightarrow$ \autoref{prop:b-nil-formally-unramified-image-relFrobenius-Noetherian.d}. Assuming \autoref{prop:b-nil-formally-unramified-image-relFrobenius-Noetherian.c}, we have $S = \im(F_\varphi) = R[S^p]$ is Noetherian. Also, since surjective maps are epimorphisms, $\varphi$ is b-nil formally unramified by \autoref{thm:b-nil-formally-unramified} hence certainly formally unramified, giving \autoref{prop:b-nil-formally-unramified-image-relFrobenius-Noetherian.d}. Finally, for \autoref{prop:b-nil-formally-unramified-image-relFrobenius-Noetherian.d}$\implies$\autoref{prop:b-nil-formally-unramified-image-relFrobenius-Noetherian.c}, consider the inclusion of Noetherian rings $R[S^p] \hookrightarrow S$. Since $\Omega_{S/R} = \Omega_{S/R[S^p]}$, this inclusion is formally unramified since $\varphi \colon R \to S$ is. But the absolute Frobenius $F_S$ of $S$ has image in $R[S^p]$. Thus, restricting the co-domain of $F_S$ we get a map $S \to R[S^p]$ that just raises elements to the $p$-th power. By construction, the composition $S \to R[S^p] \hookrightarrow S$ is equal to $F_S$ and the composition $R[S^p] \hookrightarrow S \to R[S^p]$ is equal to $F_{R[S^p]}$. Then $R[S^p] \hookrightarrow S$ is surjective by \autoref{prop:epimorphism-surjective-Noetherian}, i.e., $\im(F_\varphi) = R[S^p] = S$, giving \autoref{prop:b-nil-formally-unramified-image-relFrobenius-Noetherian.c}.
\end{proof}

\begin{remarks}
Let $R \to S$ be a map of $\mathbf{F}_p$-algebras, and denote by $R[S^p]$, as usual, the image of the relative Frobenius $F_{S/R}$.
\begin{enumerate}
    \item \autoref{prop:b-nil-formally-unramified-image-relFrobenius-Noetherian}~\autoref{prop:b-nil-formally-unramified-image-relFrobenius-Noetherian.d}$\implies$\autoref{prop:b-nil-formally-unramified-image-relFrobenius-Noetherian.c} was first observed by the first and fourth authors in \cite[Theorem~4.1]{CarvajalRojasStablerPristineMorphisms}.

    \item There exist examples where $R, S$ are Noetherian but $S^{(p)}$ is not. In fact, Radu proved in \cite[Corollaire 5]{RaduUneClasseDAnneaux} that if $R$ is a Noetherian local $\mathbf{F}_p$-algebra and $S = \widehat{R}$ is its completion, then $S^{(p)}$ is Noetherian if and only if $R$ is a G-ring.  Indeed, the relative Frobenius $F_{S/R}$ is identified with the completion of the local ring $S^{(p)}$ with respect to its maximal ideal, see for example \cite[Lemma~4.13]{DattaOlanderRelFrobIso}. Thus $S^{(p)}$ is Noetherian if and only if $F_{S/R}$ is flat: If $F_{S/R}$ is flat, hence faithfully flat, then $S^{(p)}$ is Noetherian by descent; and conversely, if $S^{(p)}$ is Noetherian, then its completion map $F_{S/R}$ is flat. By the relative version of Kunz's Theorem due to Radu and Andr\'e, see \cite{RaduUneClasseDAnneaux} and \cite{AndreHomomorphismsRegulariers}, the flatness of $F_{S/R}$ is in turn equivalent to $R \to S$ being a regular homomorphism; that is, to $R$ being a G-ring.

    \item An example where $R, S$ are Noetherian but both $S^{(p)}$ and $R[S^p]$ are not Noetherian is given by the completion map $R \to \widehat{R}$ of a Nagata local ring $(R,\m)$ of characteristic $p$ that is not a G-ring. % quasi-excellent \cite{RotthausNagatanotQuasiExcellent}, or equivalently, $R \to \widehat{R}$ is not regular. 
    Indeed, by \cite[\href{https://stacks.math.columbia.edu/tag/0BJ0}{Tag 0BJ0}]{stacks-project}, the map $R \to \widehat{R}$ has geometrically reduced fibers. This implies, by \cite[Theorem]{DumitrescuReduceness}, that $F_{\widehat{R}/R}$ is universally injective as a map of $F_*R$-algebras, hence in particular injective. Then $R[\widehat{R}^p] \cong F_*R \otimes_R \widehat{R}$ which is not Noetherian by the previous paragraph. %By \cite[Corollaire 5]{RaduUneClasseDAnneaux}, the ring $F_\ast R \otimes_R \widehat{R}$ is not Noetherian.

    \item There are also examples where $R$ and $S$ are Noetherian $\mathbf{F}_p$-algebras, $S^{(p)}$ is not Noetherian, but $R[S^p]$ is Noetherian. Indeed, let $(R,\m, \kappa)$ be any non-excellent DVR of characteristic $p$ such that the residue field $\kappa$ is $F$-finite (see \cite{DattaSmithExcellence} for such examples). Again, let $S = \widehat{R}$ be $\m$-adic completion. Then $F_{\widehat{R}/R}$ is surjective by \cite[Proposition~4.17]{DattaOlanderRelFrobIso} so $R[\widehat{R}^p] = \widehat{R}$ is Noetherian. However, $\widehat{R}^{(p)} = F_*R \otimes_R \widehat{R}$ cannot be Noetherian -- that is, $R$ cannot be a G-ring. Indeed, since $R$ is Noetherian regular local, one would then automatically get that $R$ is excellent since a local $G$-ring is automatically quasi-excellent \cite[(33.D),~Theorem~76]{MatsumuraCommutativeAlgebra} and a regular ring is universally catenary, contradicting our choice of $R$. 
    %Indeed, the relative Frobenius $F_{\widehat{R}/R}$ can be identified with the canonical map associated with the completion of the local ring $\widehat{R}^{(p)}$ with respect to the finitely generated maximal ideal $\m \widehat{R}^{(p)} \cong F_*\m \otimes_R \widehat{R}$ (see \cite[Lemma~4.13]{DattaOlanderRelFrobIso} for a more general fact about the relative Frobenius of ideal-adic completion maps). %Thus, if $\widehat{R}^{(p)}$ is Noetherian, then $F_{\widehat{R}/R}$ would be faithfully flat. However, this would imply $R \to \widehat{R}$  has geometrically regular fibers, that is, $\widehat{R}$ is a $G$-ring. Since $R$ is Noetherian regular local, one would then automatically get that $R$ is excellent since a local $G$-ring is automatically quasi-excellent \cite[(33.D),~Theorem~76]{MatsumuraCommutativeAlgebra} and a regular ring is universally catenary, contradicting our choice of $R$.
\end{enumerate}

\end{remarks}

We now show that for many types of maps between Noetherian $\mathbf{F}_p$-algebras, formally unramified is equivalent to b-nil formally unramified.

\begin{corollary}
    \label{cor:cases-bnil-formally-unramified-equivalent-Noetherian}
    Let $\varphi \colon R \to S$ be a map of $\mathbf{F}_p$-algebras such that $S$ is Noetherian. Then $\varphi$ is formally unramified if and only if it is b-nil formally unramified in the following situations:
    \begin{enumerate}
        \item $\varphi$ is $F$-finite.\label{cor:cases-bnil-formally-unramified-equivalent-Noetherian.a}
        \item $S$ is $F$-finite.\label{cor:cases-bnil-formally-unramified-equivalent-Noetherian.b}
        \item $R$ is $F$-finite.\label{cor:cases-bnil-formally-unramified-equivalent-Noetherian.c}
        %\item $\varphi$ is $F$-pure.\label{cor:cases-bnil-formally-unramified-equivalent-Noetherian.d}
        \item $F_\varphi$ is cyclically pure (e.g., if $\varphi$ is $F$-pure).\label{cor:cases-bnil-formally-unramified-equivalent-Noetherian.d}
    \end{enumerate}
    Furthermore, in all these cases, $F_\varphi$ is surjective if $\varphi$ is (b-nil) formally unramified.
\end{corollary}

Recall, a ring homomorphism $A \to B$ is \emph{cyclically pure} if, for all ideals $I$ of $A$, the induced map $A/I \to B/IB$ is injective. Thus, universally injective ring maps are cyclically pure. The converse holds in many situations when $A$ is a `nice' Noetherian ring \cite{HochsterCyclicPurity}. Note that taking $I = 0$, a cyclically pure ring map is injective.

\begin{proof}[Proof of \autoref{cor:cases-bnil-formally-unramified-equivalent-Noetherian}]
    \autoref{cor:cases-bnil-formally-unramified-equivalent-Noetherian.a} follows by \autoref{thm:F-finite-b-nil-formally-unramified} and \autoref{cor:cases-bnil-formally-unramified-equivalent-Noetherian.b} is a consequence of \autoref{cor:cases-bnil-formally-unramified-equivalent-Noetherian.a} because if $S$ is $F$-finite then $\varphi$ is $F$-finite, that is, $F_\varphi$ is a finite map. These two parts do not require $S$ to be Noetherian. However, we can also give a uniform argument to show all five cases. Namely, by \autoref{prop:b-nil-formally-unramified-image-relFrobenius-Noetherian}, it suffices to show that in all five cases, $R[S^p]$, the image of $F_\varphi$, is a Noetherian ring.

    In \autoref{cor:cases-bnil-formally-unramified-equivalent-Noetherian.a}, $R[S^p]$ is Noetherian because $R[S^p] \hookrightarrow S$ is a finite extension by the assumption that $\varphi$ is $F$-finite and it is a result of Mollier that the property of being Noetherian descends along finite extensions \cite[Part~II,~Corollaire~1.2.5]{rg71}.

    As before, \autoref{cor:cases-bnil-formally-unramified-equivalent-Noetherian.b} follows by \autoref{cor:cases-bnil-formally-unramified-equivalent-Noetherian.a}.

    If $R$ is $F$-finite, then $F_R \otimes_R \id_S \colon S \to F_*R \otimes_R S$ is finite, and hence $F_*R \otimes_R S$ is Noetherian. Since $R[S^p]$ is a homomorphic image of $F_*R \otimes_R S$ it is Noetherian as well. Thus, we get \autoref{cor:cases-bnil-formally-unramified-equivalent-Noetherian.c}.

    To show \autoref{cor:cases-bnil-formally-unramified-equivalent-Noetherian.d}, note that if $A \to B$ is a cyclically pure ring map and $B$ is Noetherian, then any ascending chain of ideals of $A$ must stabilize because the chain stabilizes after expanding to $B$ and because for any ideal $I$ of $A$ we have $IB \cap A = I$ by the cyclic purity of $A \to B$. Thus, $R[S^p] \cong F_*R \otimes_R S$ is Noetherian.

    The final assertion about surjectivity of $F_\varphi$ follows by \autoref{prop:epimorphism-surjective-Noetherian}.
\end{proof}

In fact, when the relative Frobenius is cyclically pure and the target ring is Noetherian, one can say more. This strengthens \autoref{cor:bnil-unramified-etale-equal-F-pure}.

\begin{corollary}
    \label{cor:b-nil-formally-unramified-regular-maps}
    Let $\varphi \colon R \to S$ be a map of $\mathbf{F}_p$-algebras such that $S$ is Noetherian and $F_\varphi$ is cyclically pure. Then the following are equivalent:
    \begin{enumerate}
        \item $\varphi$ is formally \'etale.\label{cor:b-nil-formally-unramified-regular-maps.a}
        \item $\varphi$ is formally unramified.\label{cor:b-nil-formally-unramified-regular-maps.b}
        \item $\varphi$ is b-nil formally unramified.\label{cor:b-nil-formally-unramified-regular-maps.c}
        \item $\varphi$ is b-nil formally \'etale.\label{cor:b-nil-formally-unramified-regular-maps.d}
    \end{enumerate}
    Thus, if $\varphi$ is a regular map of Noetherian rings, i.e., a flat map with geometrically regular fibers where $R$ and $S$ are both Noetherian, then the above four conditions are equivalent.
\end{corollary}

\begin{proof}
We always have \autoref{cor:b-nil-formally-unramified-regular-maps.d}$\implies$\autoref{cor:b-nil-formally-unramified-regular-maps.a}$\implies$\autoref{cor:b-nil-formally-unramified-regular-maps.b}. Also, \autoref{cor:b-nil-formally-unramified-regular-maps.b}$\implies$\autoref{cor:b-nil-formally-unramified-regular-maps.c} by \autoref{cor:cases-bnil-formally-unramified-equivalent-Noetherian}. Thus, it suffices to show \autoref{cor:b-nil-formally-unramified-regular-maps.c}$\implies$\autoref{cor:b-nil-formally-unramified-regular-maps.d}. By \cite[Theorem~3.21]{DattaOlanderRelFrobIso}, it is enough to show that $F_\varphi$ is an isomorphism. Assuming \autoref{cor:b-nil-formally-unramified-regular-maps.c}, $F_\varphi$ is surjective by \autoref{cor:cases-bnil-formally-unramified-equivalent-Noetherian}. Since $F_\varphi$ is cyclically pure, it is injective.

The final assertion follows because a regular homomorphism of Noetherian $\mathbf{F}_p$-algebras has faithfully flat relative Frobenius by results of Radu \cite{RaduUneClasseDAnneaux} and Andr\'e \cite{AndreHomomorphismsRegulariers}, and a faithfully flat ring map is cyclically pure. 
\end{proof}

One can also strengthen \autoref{cor:b-nil-formally-unramified-over-epi-Frobenius} with Noetherian hypotheses.

\begin{proposition}
    \label{prop:Noetherian-over-epimorphic-Frobenius}
    Let $k$ be a ring of prime characteristic $p > 0$ such that $F_k$ is an epimorphism. Let $R$ be a Noetherian $k$-algebra. Then the following are equivalent:
    \begin{enumerate}
        \item $k \to R$ is b-nil formally unramified.\label{prop:Noetherian-over-epimorphic-Frobenius.a}
        %\item $F_R$ is an epimorphism.
        \item $F_R$ is b-nil formally unramified.\label{prop:Noetherian-over-epimorphic-Frobenius.b}
        \item $F_R$ is an isomorphism, i.e., $R$ is perfect.\label{prop:Noetherian-over-epimorphic-Frobenius.c}
        \item $k \to R$ is formally unramified.\label{prop:Noetherian-over-epimorphic-Frobenius.d}
        \item $\mathbf{F}_p \to R$ is formally unramified.\label{prop:Noetherian-over-epimorphic-Frobenius.e}
        %\item $\mathbf{F}_p \to R$ is b-nil formally unramified.
    \end{enumerate}
\end{proposition}

\begin{proof}
    \autoref{prop:Noetherian-over-epimorphic-Frobenius.a}$\implies$\autoref{prop:Noetherian-over-epimorphic-Frobenius.b} follows by \autoref{cor:b-nil-formally-unramified-over-epi-Frobenius}. 
    
    \autoref{prop:Noetherian-over-epimorphic-Frobenius.b}$\implies$\autoref{prop:Noetherian-over-epimorphic-Frobenius.c}: We give two different arguments to first show that $F_R$ is surjective. Assuming \autoref{prop:Noetherian-over-epimorphic-Frobenius.b}, $F_R$ is formally unramified. Since $R$ is Noetherian, $F_R$ is surjective by \autoref{prop:epimorphism-surjective-Noetherian}. Alternatively, $F_R$ being b-nil formally unramified implies it is an epimorphism by \autoref{cor:b-nil-formally-unramified-over-epi-Frobenius}. An integral epimorphism of Noetherian rings is surjective by \autoref{prop:epic-implies-surjective}~\autoref{prop:epic-implies-surjective.c}. Having established that $F_R$ is surjective, observe that if $I$ is a radical ideal of $R$, then $I = I^{[p^e]}$ for all integers $e > 0$ because a $p^e$-th root of an element of $I$ (which always exists by the surjectivity of $F^e_R$) must be in $I$. Applying this to $\sqrt{(0)}$, we have that because $\sqrt{(0)}$ is finitely generated, it is nilpotent. Hence,
    \[
    \sqrt{(0)} = \sqrt{(0)}^{[p^e]} = (0)
    \] for all $e \gg 0$. Thus, $R$ is reduced and $F_R$ is injective as well.
    
    \autoref{prop:Noetherian-over-epimorphic-Frobenius.c}$\implies$\autoref{prop:Noetherian-over-epimorphic-Frobenius.d}: If $F_R$ is surjective then $F_{R/k}$ is surjective as well, and so, $k \to R$ is b-nil formally unramified by \autoref{thm:b-nil-formally-unramified} and hence formally unramified. 

    \autoref{prop:Noetherian-over-epimorphic-Frobenius.d}$\implies$\autoref{prop:Noetherian-over-epimorphic-Frobenius.e}: Since $F_k$ is an epimorphism, $\mathbf{F}_p \to k$ is b-nil formally unramified by \autoref{cor:b-nil-formally-unramified-over-epi-Frobenius}, and so, the composition $\mathbf{F}_p \to k \to R$ is formally unramified.

    \autoref{prop:Noetherian-over-epimorphic-Frobenius.e}$\implies$\autoref{prop:Noetherian-over-epimorphic-Frobenius.a}: Since $\mathbf{F}_p$ is $F$-finite and $R$ is Noetherian, $\mathbf{F}_p \to R$ is b-nil formally unramified by \autoref{cor:cases-bnil-formally-unramified-equivalent-Noetherian}. But $\mathbf{F}_p \to R = \mathbf{F}_p \to k \to R$, so $k \to R$ is b-nil formally unramified by \autoref{lem:basic-properties-b-nil-formally-unramified}~\autoref{cor:b-nil-formally-unramified-over-epi-Frobenius.a}.
\end{proof}

\begin{examples}
Here are some cases where absolute Frobenius being an epimorphism implies it is surjective.
\begin{enumerate}
    \item If $R$ is $F$-finite and $F_R$ is an epimorphism, then $F_R$ is surjective by \autoref{prop:epic-implies-surjective}~\autoref{prop:epic-implies-surjective.b}.
    
    \item If $R$ is a Noetherian $\mathbf{F}_p$-algebra and $F_R$ is an epimorphism, then $F_R$ is an isomorphism by \autoref{prop:Noetherian-over-epimorphic-Frobenius}. In other words, for Noetherian $R$, $F_R$ being an epimorphism is equivalent to $R$ being isomorphic to a finite product of perfect fields. %of characteristic $p$.

    \item Let $R$ be a normal $\mathbf{F}_p$-algebra, i.e., for all $\frp \in \Spec(R)$, $R_\frp$ is a domain that is integrally closed in $\Frac(R)$. If $F_R$ is an epimorphism, then $R$ is perfect. It is enough to show that $R_\frp$ is perfect. Since $F_R \otimes_R \id_{R_\frp}$ is identified with $F_{R_\frp}$, we see that $F_{R_\frp}$ is also an epimorphism by base change (\autoref{cor:epimorphisms-formally-unramified}). Thus, by passing to $R_\frp$, we immediately reduce to the case where $R$ is a domain that is integrally closed in $\Frac(R)$ and has an epimorphic $F_R$. But epimorphisms induce isomorphisms at the level of residue field extensions; see \autoref{prop:necessary-sufficient-characterization-epimorphism}. Thus, $\Frac(R)$ is perfect. But then any $x \in R$ has a $p$th root in $\Frac(R)$, hence also in $R$.
    %Since $R$ is integrally closed in $\Frac(R)$, $R^p$ is integrally closed in $\Frac(R)^p = \Frac(R)$. But $R^p \hookrightarrow R$ is an integral extension, so we must have $R^p = R$, that is, $R$ is perfect. 
\end{enumerate}
\end{examples}

\begin{example}
    \label{eg:when-is-bnil-formally-unramified-over-perfect-field-reduced}
    Let $k$ be a field of characteristic $p > 0$ and $R$ be a $k$-algebra. It is natural to ask if $k \to R$ being b-nil formally unramified implies $R$ is reduced. We discuss some examples to illustrate that this is not always the case.
    \begin{enumerate}
        \item If $k = \mathbf{F}_p$, then Bhatt constructs an example of a non-reduced semi-perfect $\mathbf{F}_p$-algebra $R$ such that $\mathbf{F}_p \to R$ is formally \'etale \cite{Bhatt_imperfect_trivial_cc} (in fact, his example satisfies the stronger property that the cotangent complex $L_{R/\mathbf{F}_p} \simeq 0$). Note that such an $R$ is necessarily non-Noetherian by \autoref{prop:Noetherian-over-epimorphic-Frobenius}. Since $F_R = F_{R/\mathbf{F}_p}$ is surjective, $\mathbf{F}_p \to R$ is b-nil formally unramified by \autoref{thm:b-nil-formally-unramified}.   
        %is $F$-finite, $\mathbf{F}_p \to R$ is b-nil formally unramified \autoref{thm:F-finite-b-nil-formally-unramified}. Thus, a b-nil formally unramified algebra over a perfect field of characteristic $p$ is not always reduced.

        \item If $k$ is perfect and $R$ is Noetherian, then $k \to R$ being b-nil formally unramified implies $R$ is reduced (in fact, $R$ is perfect) by \autoref{prop:Noetherian-over-epimorphic-Frobenius}.

        \item When $k$ is not perfect, a Noetherian b-nil formally unramified $k$-algebra $R$ may fail to be reduced. Our example comes from \cite[Example~3.2]{MUKHOPADHYAY202161}. Let $t$ be an indeterminate and let $k = \mathbf{F}_p(t)$. Let $k_{\perf}$ be the perfection of $k$ and $R = k_{\perf}[x]/(x^2)$, for an indeterminate $x$. Then $R$ is an $F$-finite Artinian local ring that is not reduced. If $\overline{x}$ is the class of $x$ in $R$, then $R$ becomes a $k$-algebra via the map
        \begin{align*}
            \phi \colon k &\longrightarrow R\\
            f(t) &\longmapsto f(t) + f'(t)\overline{x},
        \end{align*}
        where $f'(t)$ is the formal derivative of $f(t)$. Note that this is not the ``usual'' $k$-algebra structure on $R$. Since the residue field of $R$ is perfect and $R$ is complete, it has a unique coefficient field (namely, just the classes of the constant polynomials in $k_{\perf}[x]$). Note that this example also illustrates that $\im(\phi)$ is not contained in the unique coefficient of $R$. It is shown in \cite[Example~3.2]{MUKHOPADHYAY202161} that $\phi$ is formally unramified, and in fact, that the relative Frobenius of $\phi$ is surjective. Hence $\phi$ is also b-nil formally unramified by \autoref{thm:b-nil-formally-unramified}. %\autoref{cor:cases-bnil-formally-unramified-equivalent-Noetherian}.
    \end{enumerate}
\end{example}

In \autoref{sec:epimorphic-Frobenius-not-surjective}, we construct an $\mathbf{F}_p$-algebra $R$, which is necessarily non-Noetherian, such that $F_R$ is an epimorphism but not surjective.

\section{A Ring with Epimorphic but not Surjective Absolute Frobenius}
\label{sec:epimorphic-Frobenius-not-surjective}

Fix a prime number $p > 0$. Given an $\mathbf{F}_p$-algebra $A$ and elements $a_1, \dots , a_n \in A$, the ring obtained from $A$ by freely adjoining a compatible system of $p$-power roots of $a_1, \dots , a_n$ is
$$
A' = \operatorname{colim}_i A[z_1, \dots , z_n]/\langle z_j^{p^i}-a_j \rangle
$$
where the colimit is over transition maps which are the identity on $A$ and send $z_j \mapsto z_j^p$. The ring $A'$ has the following universal property: Given an $A$-algebra $B$, specifying a homomorphism of $A$-algebras $A' \to B$ is the same as giving elements $b_{jk} \in B$ for $j = 1, \dots , n$ and $k \geq 1$ with the property that $b_{j,1} $ is the image of $a_j$ in $B$ and $b_{j,k+1}^p = b_{j,k}$ for all $k \geq 1$ and $j = 1, \dots , n$.

For $n \geq 2$, let $A_n$ be the ring obtained from the polynomial ring
$$
\mathbf{F}_p[x_1, \dots , x_{2^n-1}]
$$
by freely adjoining compatible systems of $p$-power roots to the elements $x_{i}$ with $i$ even and to the elements $x_ix_{i+1}$ for arbitrary $i \leq 2^n-2$. Thus, for an $\mathbf{F}_p$-algebra $B$, specifying a ring map $A_n \to B$ is the same as giving $2^n-1$ elements $b_1, \dots, b_{2^n - 1} \in B$ together with compatible systems of $p$-power roots for each of the elements
$$
b_1b_2, b_2, b_2b_3, b_3b_4, b_4, b_4b_5, b_5b_6, b_6, \dots, b_{2^n-2},b_{2^n-2}b_{2^n-1}.
$$
Define ring maps $A_n \to A_{n+1}$ such that
$x_i \mapsto x_{2i}$ for $i$ even, and $x_i \mapsto x_{2i-1}x_{2i}x_{2i+1}$ for $i$ odd. That is, 
$$x_1 \mapsto x_1x_2x_3, x_2 \mapsto x_4, x_3 \mapsto x_5x_6x_7, x_4 \mapsto x_8, \cdots  $$
There exist such maps because the product of any 4 consecutive variables $x_ix_{i+1}x_{i+2}x_{i+3}$ in $A_{n+1}$ has a compatible system of $p$-power roots. Define
$$
A = \operatorname{colim}_n A_n.
$$

\begin{claim}
    The absolute Frobenius of $A$ is an epimorphism. 
\end{claim}
\begin{proof}[Proof of claim]
Suppose $n \geq 2$, $B$ is a ring, and there exist two commutative triangles

\[\begin{tikzcd}
	{A_n} & {} & {A_{n+1}} && {A_n} & {} & {A_{n+1}} \\
	& B &&&& B
	\arrow[from=1-1, to=1-3]
	\arrow["{\varphi_n}"', from=1-1, to=2-2]
	\arrow["{\varphi_{n+1}}", from=1-3, to=2-2]
	\arrow[from=1-5, to=1-7]
	\arrow["{\psi_n}"', from=1-5, to=2-6]
	\arrow["{\psi_{n+1}}", from=1-7, to=2-6]
\end{tikzcd}\]
of rings with the property that $\varphi_i|_{A_i^p} = \psi_i|_{A_i^p}$ for $i = n, n+1$. It is enough to show $\varphi_n = \psi_n$. 

Certainly,
\begin{align*}
\varphi_n\Big(x_{2i}^{1/p^j}\Big) &= \psi_n\Big(x_{2i}^{1/p^j}\Big), \hspace{5 em} i < 2^{n-1}, j \geq 0, \\
\varphi_{n+1}\Big(x_{2i}^{1/p^j}\Big) &= \psi_{n+1}\Big(x_{2i}^{1/p^j}\Big), \hspace{4 em} i < 2^n, j \geq 0,
\end{align*}
as the inputs are $p^{th}$ powers. For the same reason,
\begin{align*}
    \varphi_n\Big((x_ix_{i+1})^{1/p^j}\Big) &= \psi_n\Big((x_ix_{i+1})^{1/p^j}\Big), \hspace{5 em} i < 2^n-1, j \geq 0, \\
    \varphi_{n+1}\Big((x_ix_{i+1})^{1/p^j}\Big) &= \psi_{n+1}\Big((x_ix_{i+1})^{1/p^j}\Big), \hspace{4 em} i < 2^{n+1}-1, j \geq 0.
\end{align*}

To show $\varphi_n = \psi_n$, it therefore suffices to show $\varphi_n(x_i) = \psi_n(x_i)$ for all odd $i$. For such $i$ we have
\begin{align*}
\varphi_n(x_i) &= \varphi_{n+1}(x_{2i-1}x_{2i}x_{2i+1}) = \varphi_{n+1}(x_{2i-1})\varphi_{n+1}(x_{2i}x_{2i+1})\\ 
&= \varphi_{n+1}(x_{2i-1})\psi_{n+1}(x_{2i}x_{2i+1})
= \varphi_{n+1}(x_{2i-1})\psi_{n+1}(x_{2i})\psi_{n+1}(x_{2i+1})\\ 
&= \varphi_{n+1}(x_{2i-1})\varphi_{n+1}(x_{2i})\psi_{n+1}(x_{2i+1}) 
= \varphi_{n+1}(x_{2i-1}x_{2i})\psi_{n+1}(x_{2i+1}) \\
&= \psi_{n+1}(x_{2i-1}x_{2i})\psi_{n+1}(x_{2i+1}) 
= \psi_{n+1}(x_{2i-1}x_{2i}x_{2i+1}) \\
&= \psi_{n}(x_{i});
\end{align*}
as needed. 
\end{proof}
\begin{claim}
    The Frobenius of $A$ is not surjective. 
\end{claim}

The element $x_1x_2x_3 \in A_2$ has image $x_1\cdots x_{2^n-1}$ in $A_n$ for $n \geq 2$. It suffices to show that this element is not a $p^{th}$ power in $A_n$. The ring $A_n$ is the colimit over $i$ of the rings
\begin{align*}
 B_{n, i} &= \dfrac{ \mathbf{F}_p[x_1, \dots , x_{2^n-1}, z_2, z_4, \dots , z_{2^n-2}, w_{1,2}, w_{2,3}, \dots , w_{2^n-2, 2^n-1}]}{\Big\langle z_{2j}^{p^i}-x_{2j}, w_{j, j+1}^{p^i} - x_jx_{j+1} \Big\rangle }  \\
&= \dfrac{ \mathbf{F}_p[x_1, z_2, x_3, \dots , z_{2^n-2}, x_{2^n-1}, w_{1,2}, w_{2,3}, \dots ,  w_{2^n-2, 2^n-1}]}{\Big\langle w_{2j+1, 2j+2}^{p^i} - x_{2j+1}z_{2j+2}^{p^i} , w_{2j, 2j+1}^{p^i} - z_{2j}^{p^i}x_{2j+1}\Big\rangle } ,
\end{align*}
so it suffices to show the element
\begin{equation}
\label{equn-theelement}
x_1 z_2^{p^i}x_3z_4^{p^i} \cdots z_{2^n-2}^{p^i}x_{2^n-1},
\end{equation}
which maps to $x_1\cdots x_{2^n-1}$ in $A_n$, is not a $p^{th}$ power in $B_{n, i}$ for any $i$. The ring $B_{n, i}$ is a free module over the polynomial ring $\mathbf{F}_p[x_1, z_2, \dots , z_{2^n-2}, x_{2^n-1}]$ with a basis consisting of the monomials 
\[
w^e = w_{1,2}^{e_{1,2}} w_{2,3}^{e_{2,3}} \cdots w_{2^n-2, 2^n-1}^{e_{2^n-2, 2^n-1}}
\] where $e = (e_{1,2}, e_{2,3}, \dots )$ is a multi-index such that $0 \leq e_{j, j+1} < p^i$ for all $j$. Let $E$ be the set of all such multi-indices. Suppose the element \autoref{equn-theelement} is a $p^{th}$ power in $B_{n,i}$. Then we can write
$$
x_1 z_2^{p^i}x_3z_4^{p^i} \cdots z_{2^n-2}^{p^i}x_{2^n-1} = \bigg(\sum _{e \in E}f_e w^e\bigg)^p = \sum_{e \in E} f_e^p w^{pe}
$$
with $f_e \in \mathbf{F}_p[x_1, z_2, \dots , z_{2^n-2}, x_{2^n-1}]$. Since the left hand side has no terms involving $w_{j,j
+1}$, which corresponds on the right hand side to those multi-indices $e$ such that $pe_{j,j+1} \equiv 0 \pmod {p^i}$, we get 
%comparing coefficients of the basis element $1 = w_{1,2}^0w_{2,3}^0 \cdots w_{2^n-2, 2^n-1}^0$ on both sides gives
\[
x_1 z_2^{p^i}x_3z_4^{p^i} \cdots z_{2^n-2}^{p^i}x_{2^n-1} = \sum_{e \in E'}f_e^p w^{pe}
\]
where $E' \subset E$ is the set of multi-indices $e$ with $pe_{j, j+1} \equiv 0 \pmod {p^i}$ for all $j$, or equivalently, such that there exists a unique multi-index $k = (k_{1,2}, k_{2,3}, \dots )$ with $0 \leq k_{j, j+1} < p$ and $p^{i-1} k_{j, j+1} = e_{j, j+1}$ for all $j \leq 2^n-2$. Thus, we can write 
$$
x_1 z_2^{p^i}x_3z_4^{p^i} \cdots z_{2^n-2}^{p^i}x_{2^n-1} = \sum_{k \in K}f_{p^{i-1}k}^p w^{p^ik}
$$
where $K$ is the set of multi-indices $k$ with $0 \leq k_{j, j+1} < p$ for all $j$. By the presentation of $B_{n,i}$, this gives
\begin{align*}
&x_1 z_2^{p^i}x_3z_4^{p^i} \cdots z_{2^n-2}^{p^i}x_{2^n-1} = \sum_{k \in K}f_{p^{i-1}k}^p w^{p^ik}\\  %\sum_{k \in K}f_{p^{i-1}k}^p (x_1x_2)^{k_{1,2}} (x_2x_3)^{k_{2,3}} \cdots( x_{2^n-2} x_{2^n-1})^{k_{2^n-2, 2^n-1}} \\
&= \sum _{k \in K} f_{p^{i-1}k}^p x_1^{k_{1,2}}z_2^{p^i(k_{1,2}+k_{2,3})}x_3^{k_{2,3}+k_{3,4}}z_4^{p^i(k_{3,4}+k_{4,5})} \cdots z_{2^n-2}^{p^i(k_{2^n-3, 2^n-2} + k_{2^n-2,2^n-1})}  x_{2^n-1}^{k_{2^n-2, 2^n-1}}.
\end{align*}
We can regard this last equality as one in the ring $\mathbf{F}_p[x_1,z_2,\dots,z_{2^n-2},x_{2^n-1}]$. Comparing the monomials occurring on both sides, we see that the only terms on the right hand side that contribute to the sum for a multi-index $k \in K$ must have $k_{1,2} = 1$, which further implies
$$k_{2,3} = 0, k_{3,4} = 1, k_{4,5} = 0, k_{5, 6} = 1, \dots , k_{2^n-3, 2^n-2} = 1, k_{2^n-2, 2^n-1} = 0.$$
That is, letting $k^0 = (k^0_{1,2}, k^0_{2,3}, \dots )$ be the multi-index $(1, 0, 1, 0, \dots , 1, 0)$, we must have
\begin{equation}
\label{equn-comparemonomials}
x_1 z_2^{p^i}x_3z_2^{p^i} \cdots z_{2^n-2}^{p^i}x_{2^n-1} %= f_{p^{i-1}k^0}w^{p^ik^0} =
=f_{p^{i-1}k^0}^p x_1z_2^{p^i}x_3z_4^{p^i} \cdots  z_{2^n-2}^{p^i}x_{2^n-1}^{0}.
\end{equation}
However, this equation has no solutions with $f_{p^{i-1}k^0} \in \mathbf{F}_p[x_1, z_2, \dots , z_{2^n-2}, x_{2^n-1}]$ as any monomial occurring in the right hand side of \autoref{equn-comparemonomials} has $x_{2^n-1}$ raised to a power divisible by $p$. This is a contradiction.

\begin{remark}
    The above example settles a question raised by Matthieu Romagny \cite{MOFrobeniusepisurjective} on Mathoverflow. 
\end{remark}

\begin{proposition}
    \label{prop:b-nil-formally-unramified-non-surjective-relative-Frobenius}
    There exists a ring map $\mathbf{F}_p \to A$ that is b-nil formally unramified, but such that the relative Frobenius $F_{A/\mathbf{F}_p}$ is not surjective.
\end{proposition}

\begin{proof}
    Take $A$ to be the example above; that is, such that $F_A$ is an epimorphism but not surjective. Then $\mathbf{F}_p \to A$ is b-nil formally unramified by \autoref{cor:b-nil-formally-unramified-over-epi-Frobenius}, but $F_{A/\mathbf{F}_p}$, which can be identified with $F_A$, is not surjective.
\end{proof}

\section{A Formally Unramified Map that is not B-nil Formally Unramified}
\label{sec:formally-unramified-not-bnil}

In this section, we show that there is a formally unramified ring map in positive prime characteristic that is not b-nil-formally unramified. The example is a very minor modification of an example by Gabber of a $\mathbf{Q}$-algebra that is formally unramified but not reduced, which we learned about from the paper by Mukhopadhyay–Smith \cite{MUKHOPADHYAY202161}.

 Let $p> 0$ be a prime number and $k$ be a field of characteristic $p$. Consider the category $\mathcal{C}$ of local $k$-algebras $(A,\mathfrak{m},k)$ such that $\mathfrak{m}^{[p]}=0$. Alternatively, this is the category of $k$-algebras $A$ with an ideal $I \subset A$ such that: 
 \begin{enumerate}[(1)]
 \item $k \to A/I$ is an isomorphism, and
 \item $I^{[p]} = 0$.
 \end{enumerate}
In particular, any $k$-algebra homomorphism between two such $k$-algebras is local. Note also that if $(A,\m,k) \in \mathcal{C}$, then $k \hookrightarrow A$ is an integral extension because $A^p = k^p \subset k \subset A$. Indeed, if $a \in A$, then we can write $a = u + x$ where $u \in k$ and $x \in \m$. Then $a^p = u^p + x^p = u^p \in k^p$.

\begin{lemma}
    The category $\mathcal{C}$ is closed under tensor products, filtered colimits, and non-zero quotients in the category of $k$-algebras. \qed
\end{lemma}

\begin{lemma}
\label{lemma-epicriterion}
    If $A \in \mathcal{C}$, the relative Frobenius $A^{(p)} \to A$ of $k \to A$ is an epimorphism of rings if and only if the inclusion $k \to A$ is bijective. 
\end{lemma}

\begin{proof}
    The map $F_{A/k} : A^{(p)} \to A$ is an epimorphism if and only if the inclusion of its image $k[A^p] \hookrightarrow A$ is an epimorphism by \autoref{lem:epimorphism-extension-reduction}. We have $A^p = k^p \subset A$ hence $k[A^p] = k$ as a subset of $A$, so this is equivalent to the inclusion $k \to A$ being an epimorphism. Now we conclude by \autoref{cor:epimorphisms-from-fields}.
\end{proof}

Let us assume $k$ is a field of characteristic $p \geq 7$. Let $A_0 \coloneqq  k\llbracket x,y \rrbracket/\langle\frac{\partial f}{\partial x}, \frac{\partial f}{\partial y}\rangle$ where $f(x,y) = x^2y^2 +x^5+y^5$. Then:

\begin{enumerate}
    \item $A_0$ is a finite local $k$-algebra with residue field $k$ and maximal ideal $\mathfrak{m} = \langle x, y \rangle $.
    \item $f^2 = 0$ in $A_0$ but $f \neq 0$ in $A_0$.
    \item $\mathfrak{m}^ {[p]}=0$%$F(\mathfrak{m}_{A_0})A_0 = 0$. 
    \item In particular, $A_0 \in \mathcal{C}$.
\end{enumerate}

The proof of the first two items is given in \cite[Lemma 2.4]{MUKHOPADHYAY202161} (the same proof works in any characteristic $\geq 7$). Let us prove (c). It suffices to show $x^7 = 0 = y^7$. By symmetry, it is enough to show $y^7 = 0$. In the ring $A_0$, we have:
$$
\partial f/\partial x = 2xy^2 + 5x^4= 0
$$
so $xy^2 = ax^4$ where $a = -5/2$. By symmetry, also $yx^2 = ay^4$. Next, we have
$$
xy^3 = (xy^2)y = ax^4y = ax^2(yx^2) = (ax^2) (ay^4) = a^2(x^2y)y^3 = a^3y^7.
$$
Then
$$
y^4x^2 = (a^{-1}yx^2)x^2 = a^{-1}x^4y = a^{-1}(a^{-1}xy^2)y = a^{-2}xy^3 = ay^7.
$$
Also,
$$
x^2y^3 = x(xy^3) = a^3xy^7 = a^3(xy^3)y^4 = a^6y^{11},
$$
so combining gives
$$
y^7 = a^{-1}y^4x^2 = a^{-1}(x^2y^3)y = a^5y^{12}.
$$
Thus $y^7 = a^5y^{12}$ or $y^7(1-a^5y^5) = 0$, but $1 - a^5y^5 \in A_0^\times$, hence $y^7= 0$ as needed. 

Compare the following to \cite[Lemma 2.5]{MUKHOPADHYAY202161}.

\begin{lemma}
    Continue to assume $p \geq 7$ and let $(A,\mathfrak{m}) \in \mathcal{C}$ be finite dimensional. Then there is an injective morphism $A \hookrightarrow A'$ in $\mathcal{C}$ such that:
    \begin{enumerate}
        \item $A'$ is finite-dimensional.
        \item $\Omega_{A/k} \to \Omega_{A'/k}$ is zero.
    \end{enumerate}
\end{lemma}

\begin{proof}
    It suffices to show that given any element $x \in \mathfrak{m}$, there is an injective morphism $A \hookrightarrow A'$ satisfying (a) and $dx = 0$ in $A'$. Indeed, letting $x_1, \dots , x_e \in \mathfrak{m}$ be a $k$-basis, we can solve the problem first for $x_1$, then for $x_2 \in A'$, and so forth, and the composition \[
    A \to A' \to A'' \to \cdots \to A ^{(e)}
    \] works because $dx = 0$ in $\Omega_{A^{(e)}/k}$ for every $x \in \mathfrak{m}_A$, and thus also for every element of $A$ since every such element can be written as $u +x$ where $u \in k$ and $x \in \mathfrak{m}_A$.

    Thus, let $x \in \mathfrak{m}$. Let $t$ be the least integer such that $x^t = 0$. Note that $t \leq p$. Set \[
    B_t \coloneqq A_0 \otimes _k A_0 \otimes_k \cdots \otimes_k A_0 = A_0^{\otimes(t-1)},
    \] and set 
    \[
    g \coloneqq f \otimes 1 \otimes \cdots \otimes 1 + 1 \otimes f \otimes 1 \otimes \cdots \otimes 1 + \cdots + 1 \otimes \cdots \otimes 1 \otimes f,
    \]
    where $f \in A_0$ is the element specified above. Define: 
    \[
    A' = (A \otimes _k B_t)/ (x \otimes 1 - 1 \otimes g).
    \] Then $A' \in \mathcal{C}$ and $A'$ is finite dimensional. 
    Also, 
    \[
    d(x \otimes 1) = d(1 \otimes g) = \sum d(1 \otimes 1 \otimes \cdots \otimes f \otimes \cdots \otimes 1) = 0,
    \] as $df = 0$ in $\Omega_{A/k}$. Finally, the map $A' \to A$ is injective. Note that in the ring $B_t$, since $f^2 = 0$, we have $g^t = 0$ but $g^{t-1} = (t-1)! f \otimes f \otimes \cdots \otimes f \neq 0$ since $t \leq p$. Then as in \cite{MUKHOPADHYAY202161}, write $A_t = k[\epsilon]/(\epsilon^t)$, and we may view $A$ and $B_t$ as $A_t$-algebras via $\epsilon \mapsto x$ and $\epsilon \mapsto g$, and these are both injective homomorphisms. The natural map
    $$
    A \otimes _k B_t \to A \otimes _{A_t} B_t
    $$
    factors through $A'$, and so it suffices to show the composition $A \to A \otimes _k B_t \to A \otimes _{A_t} B_t$ is injective. Now  $A_t$ is injective as an $A_t$-module ($A_t$ being Gorenstein local of dimension $0$), so $A_t \to B_t$ splits in the category of $A_t$-modules. Tensoring over $A_t$ with $A$ shows that $A \to A \otimes _{A_t} B_t$ is split injective as well, completing the proof. 
\end{proof}

Let $A_0$ be as above. 
Set $A_1 = (A_0)', A_2 = (A_1)'$,
and so forth. We get a chain of injective maps
$$
A_0 \to A_1 \to \cdots 
$$
in $\mathcal{C}$. Let $A \in \mathcal{C}$ be the colimit. Then we see that $\Omega_{A/k} = \operatorname{colim}_i \Omega_{A_i/k} = 0$ since the transition maps are zero. However, $k \subsetneq A$ as $k \subsetneq A_0 \subset  A$. It follows from \autoref{lemma-epicriterion} that the relative Frobenius of $k \to A$ is not an epimorphism, so $k \to A$ is not b-nil-formally unramified. Since $k \to A$ is an integral extension, $A$ has to be non-Noetherian (equivalently, non-Artinian) by \autoref{cor:integral-epics-to-dim0-Noetherian}.

\section{Acknowledgements}

The third author is grateful to Matthew Morrow, Benjamin Antieau, and other participants of the 2019 Arizona Winter School for many interesting conversations about lifting properties.

\bibliographystyle{alpha}
\bibliography{main,preprints}

@preamble{"\def\cfudot#1{\ifmmode\setbox7\hbox{$\accent"5E#1$}\else \setbox7\hbox{\accent"5E#1}\penalty 10000\relax\fi\raise 1\ht7 \hbox{\raise.1ex\hbox to 1\wd7{\hss.\hss}}\penalty 10000 \hskip-1\wd7\penalty 10000\box7} "}

@book{RaynaudHenselian,
    AUTHOR = {Raynaud, Michel},
     TITLE = {Anneaux locaux hens\'eliens},
    SERIES = {Lecture Notes in Mathematics},
    VOLUME = {Vol. 169},
 PUBLISHER = {Springer-Verlag, Berlin-New York},
      YEAR = {1970},
     PAGES = {v+129},
   MRCLASS = {13.95},
  MRNUMBER = {277519},
MRREVIEWER = {J.-P.\ Lafon},
}

@article {AndreHomologyFrobenius,
    AUTHOR = {Andr\'e, M.},
     TITLE = {Homologie de {F}robenius},
   JOURNAL = {Math. Ann.},
  FJOURNAL = {Mathematische Annalen},
    VOLUME = {290},
      YEAR = {1991},
    NUMBER = {1},
     PAGES = {129--181},
      ISSN = {0025-5831,1432-1807},
   MRCLASS = {13D03 (13N05 18G99)},
  MRNUMBER = {1107666},
MRREVIEWER = {Luchezar\ L.\ Avramov},
       DOI = {10.1007/BF01459241},
       URL = {https://doi.org/10.1007/BF01459241},
}

@article {HashimotoFfinite,
    AUTHOR = {Hashimoto, Mitsuyasu},
     TITLE = {{$F$}-finiteness of homomorphisms and its descent},
   JOURNAL = {Osaka J. Math.},
  FJOURNAL = {Osaka Journal of Mathematics},
    VOLUME = {52},
      YEAR = {2015},
    NUMBER = {1},
     PAGES = {205--213},
      ISSN = {0030-6126},
   MRCLASS = {13A35 (13E15 13F40)},
  MRNUMBER = {3326608},
MRREVIEWER = {Reza\ Naghipour},
       URL = {http://projecteuclid.org/euclid.ojm/1427202878},
}

@book {SamuelSeminar,
     TITLE = {S\'eminaire d'{A}lg\`ebre {C}ommutative dirig\'e{} par
              {P}ierre {S}amuel: 1967/1968. {L}es \'epimorphismes d'anneaux},
 PUBLISHER = {Secr\'etariat math\'ematique, Paris},
      YEAR = {1968},
     PAGES = {ii+111 pp. (not consecutively paged)},
   MRCLASS = {13.00},
  MRNUMBER = {245561},
MRREVIEWER = {M.\ Nagata},
}

@article {EGAIV_I,
    AUTHOR = {Grothendieck, A.},
     TITLE = {\'{E}l\'ements de g\'eom\'etrie alg\'ebrique. {IV}. \'{E}tude
              locale des sch\'emas et des morphismes de sch\'emas. {I}},
   JOURNAL = {Inst. Hautes \'Etudes Sci. Publ. Math.},
  FJOURNAL = {Institut des Hautes \'Etudes Scientifiques. Publications
              Math\'ematiques},
    NUMBER = {20},
      YEAR = {1964},
     PAGES = {259},
      ISSN = {0073-8301,1618-1913},
   MRCLASS = {14.05 (13.10)},
  MRNUMBER = {173675},
MRREVIEWER = {H.\ Hironaka},
       URL = {http://www.numdam.org/item/PMIHES_1964__20__5_0},
}

@article{DumitrescuReduceness,
  author =        {Dumitrescu, Tiberiu},
  journal =       {Comm. Algebra},
  number =        {5},
  pages =         {1787--1795},
  title =         {Reducedness, formal smoothness and approximation in
                   characteristic {$p$}},
  volume =        {23},
  year =          {1995},
  doi =           {10.1080/00927879508825309},
  issn =          {0092-7872},
  url =           {https://doi-org.proxy2.cl.msu.edu/10.1080/
                  00927879508825309},
}

@book{SGA5,
     TITLE = {Cohomologie {$l$}-adique et fonctions {$L$}},
    SERIES = {Lecture Notes in Mathematics},
    VOLUME = {Vol. 589},
      NOTE = {S\'eminaire de G\'eometrie Alg\'ebrique du Bois-Marie
              1965--1966 (SGA 5),
              Edit\'e{} par Luc Illusie},
 PUBLISHER = {Springer-Verlag, Berlin-New York},
      YEAR = {1977},
     PAGES = {xii+484},
      ISBN = {3-540-08248-4},
   MRCLASS = {14F20},
  MRNUMBER = {491704},
MRREVIEWER = {James\ Milne},
}

@incollection {DattaSmithExcellence,
    AUTHOR = {Datta, Rankeya and Smith, Karen E.},
     TITLE = {Excellence in prime characteristic},
 BOOKTITLE = {Local and global methods in algebraic geometry},
    SERIES = {Contemp. Math.},
    VOLUME = {712},
     PAGES = {105--116},
 PUBLISHER = {Amer. Math. Soc., [Providence], RI},
      YEAR = {[2018] \copyright 2018},
   MRCLASS = {13A35},
  MRNUMBER = {3832401},
MRREVIEWER = {Yongwei Yao},
       DOI = {10.1090/conm/712/14344},
       URL = {https://doi-org.proxy.cc.uic.edu/10.1090/conm/712/14344},
}

@article {HashimotoFpure,
    AUTHOR = {Hashimoto, Mitsuyasu},
     TITLE = {{$F$}-pure homomorphisms, strong {$F$}-regularity, and
              {$F$}-injectivity},
   JOURNAL = {Comm. Algebra},
  FJOURNAL = {Communications in Algebra},
    VOLUME = {38},
      YEAR = {2010},
    NUMBER = {12},
     PAGES = {4569--4596},
      ISSN = {0092-7872},
   MRCLASS = {13A35 (14L30)},
  MRNUMBER = {2764840},
MRREVIEWER = {Geoffrey D. Dietz},
       DOI = {10.1080/00927870903431241},
       URL = {https://doi-org.proxy2.cl.msu.edu/10.1080/00927870903431241},
}

@article {BrennerJeffriesNunezBetancourtQuantifyingSingularities,
    AUTHOR = {Brenner, Holger and Jeffries, Jack and N{\'u}{\~n}ez-Betancourt,
              Luis},
     TITLE = {Quantifying singularities with differential operators},
   JOURNAL = {Adv. Math.},
  FJOURNAL = {Advances in Mathematics},
    VOLUME = {358},
      YEAR = {2019},
     PAGES = {106843, 89},
      ISSN = {0001-8708,1090-2082},
   MRCLASS = {14F10 (13A35 16S32)},
  MRNUMBER = {4020453},
MRREVIEWER = {Ana\ Bravo},
       DOI = {10.1016/j.aim.2019.106843},
       URL = {https://doi.org/10.1016/j.aim.2019.106843},
}

@article{EGAIV,
	Author = {Grothendieck, A.},
	Fjournal = {Institut des Hautes \'Etudes Scientifiques. Publications Math\'ematiques},
	Issn = {0073-8301},
	Journal = {Inst. Hautes \'Etudes Sci. Publ. Math.},
	Mrclass = {14.55},
	Mrnumber = {MR0238860 (39 \#220)},
	Mrreviewer = {J. P. Murre},
	Number = {32},
	Pages = {361},
	Title = {\'{E}l\'ements de g\'eom\'etrie alg\'ebrique. {IV}. \'{E}tude locale des sch\'emas et des morphismes de sch\'emas {IV}},
	Year = {1967}
}

@article{HochsterCyclicPurity,
	Author = {Hochster, Melvin},
	Fjournal = {Transactions of the American Mathematical Society},
	Issn = {0002-9947},
	Journal = {Trans. Amer. Math. Soc.},
	Mrclass = {13D99},
	Mrnumber = {MR0463152 (57 \#3111)},
	Mrreviewer = {T. Albu},
	Number = {2},
	Pages = {463--488},
	Title = {Cyclic purity versus purity in excellent {N}oetherian rings},
	Volume = {231},
	Year = {1977}
}

@book{MatsumuraCommutativeRingTheory,
	Address = {Cambridge},
	Author = {Matsumura, Hideyuki},
	Edition = {Second},
	Isbn = {0-521-36764-6},
	Mrclass = {13-01},
	Mrnumber = {MR1011461 (90i:13001)},
	Note = {Translated from the Japanese by M. Reid},
	Pages = {xiv+320},
	Publisher = {Cambridge University Press},
	Series = {Cambridge Studies in Advanced Mathematics},
	Title = {Commutative ring theory},
	Volume = {8},
	Year = {1989}
}

@book{MatsumuraCommutativeAlgebra,
	Author = {Matsumura, Hideyuki},
	Edition = {Second},
	Isbn = {0-8053-7026-9},
	Mrclass = {13-02},
	Mrnumber = {MR575344 (82i:13003)},
	Mrreviewer = {R. M. Fossum},
	Pages = {xv+313},
	Publisher = {Benjamin/Cummings Publishing Co., Inc., Reading, Mass.},
	Series = {Mathematics Lecture Note Series},
	Title = {Commutative algebra},
	Volume = {56},
	Year = {1980}
}

@article {MukhopadhyaySmithTrivialDiffOps,
    AUTHOR = {Mukhopadhyay, Alapan and Smith, Karen E.},
     TITLE = {Some algebras with trivial rings of differential operators},
   JOURNAL = {Int. Math. Res. Not. IMRN},
  FJOURNAL = {International Mathematics Research Notices. IMRN},
      YEAR = {2025},
    NUMBER = {14},
     PAGES = {Paper No. rnaf213, 18},
      ISSN = {1073-7928,1687-0247},
   MRCLASS = {13N10 (13A35 14F10)},
  MRNUMBER = {4935631},
       DOI = {10.1093/imrn/rnaf213},
       URL = {https://doi.org/10.1093/imrn/rnaf213},
}

@article{OhiDirectSummand,
    AUTHOR = {Ohi, Takeo},
     TITLE = {Direct summand conjecture and descent for flatness},
   JOURNAL = {Proc. Amer. Math. Soc.},
  FJOURNAL = {Proceedings of the American Mathematical Society},
    VOLUME = {124},
      YEAR = {1996},
    NUMBER = {7},
     PAGES = {1967--1968},
      ISSN = {0002-9939,1088-6826},
   MRCLASS = {13B02},
  MRNUMBER = {1317044},
       DOI = {10.1090/S0002-9939-96-03270-4},
       URL = {https://doi.org/10.1090/S0002-9939-96-03270-4},
}

@article{rg71,
  author =        {Raynaud, M. and Gruson, L.},
  journal =       {Invent. Math.},
  pages =         {1--89},
  title =         {Crit\`eres de platitude et de projectivit\'e.
                   {T}echniques de ``platification'' d'un module},
  volume =        {13},
  year =          {1971},
  doi =           {10.1007/BF01390094},
  issn =          {0020-9910},
  url =           {http://dx.doi.org/10.1007/BF01390094},
}

@article {YekutieliExplicitConstructionResidueComplex,
    AUTHOR = {Yekutieli, Amnon},
     TITLE = {An explicit construction of the {G}rothendieck residue
              complex},
      NOTE = {With an appendix by Pramathanath Sastry},
   JOURNAL = {Ast\'erisque},
  FJOURNAL = {Ast\'erisque},
    NUMBER = {208},
      YEAR = {1992},
     PAGES = {127},
      ISSN = {0303-1179,2492-5926},
   MRCLASS = {14F10 (12J10 14B15)},
  MRNUMBER = {1213064},
MRREVIEWER = {Reinhold\ H\"ubl},
}

@misc{stacks-project,
	Author = {The {Stacks Project Authors}},
	Howpublished = {\url{http://stacks.math.columbia.edu}},
	Shorthand = {Stacks},
	Title = {{\itshape Stacks Project}}
}

@article{RaduUneClasseDAnneaux,
	Author = {Radu, Nicolae},
	Fjournal = {Revue Roumaine de Math\'ematiques Pures et Appliqu\'ees. Romanian Journal of Pure and Applied Mathematics},
	Issn = {0035-3965},
	Journal = {Rev. Roumaine Math. Pures Appl.},
	Mrclass = {13E05 (13F40 13H05)},
	Mrnumber = {1172271 (93g:13014)},
	Mrreviewer = {Takashi Harase},
	Number = {1},
	Pages = {79--82},
	Title = {Une classe d'anneaux noeth\'eriens},
	Volume = {37},
	Year = {1992}
}

@MISC {MOFrobeniusepisurjective,
    TITLE = {If the {F}robenius endomorphism of a characteristic $p$ ring is epimorphic, is it surjective?},
    AUTHOR = {Matthieu Romagny},
    HOWPUBLISHED = {MathOverflow},
    NOTE = {URL:https://mathoverflow.net/q/400746 (version: 2021-07-31)},
    EPRINT = {https://mathoverflow.net/q/400746},
    URL = {https://mathoverflow.net/q/400746}
}

@article{AndreHomomorphismsRegulariers,
	Author = {Andr{\'e}, Michel},
	Coden = {CASMEI},
	Fjournal = {Comptes Rendus de l'Acad\'emie des Sciences. S\'erie I. Math\'ematique},
	Issn = {0764-4442},
	Journal = {C. R. Acad. Sci. Paris S\'er. I Math.},
	Mrclass = {13D03 (13A35)},
	Mrnumber = {1214408 (94e:13026)},
	Mrreviewer = {Marco Fontana},
	Number = {7},
	Pages = {643--646},
	Title = {Homomorphismes r\'eguliers en caract\'eristique {$p$}},
	Volume = {316},
	Year = {1993}
}

@article{AndreDirectsummandconjecture,
	Author = {Andr\'{e}, Yves},
	Doi = {10.1007/s10240-017-0097-9},
	Fjournal = {Publications Math\'ematiques. Institut de Hautes \'Etudes Scientifiques},
	Issn = {0073-8301},
	Journal = {Publ. Math. Inst. Hautes \'{E}tudes Sci.},
	Mrclass = {13 (16 17 18)},
	Mrnumber = {3814651},
	Pages = {71--93},
	Title = {La conjecture du facteur direct},
	Url = {https://doi.org/10.1007/s10240-017-0097-9},
	Volume = {127},
	Year = {2018}
}

@article{MUKHOPADHYAY202161,
title = {Reducedness of formally unramified algebras over fields},
journal = {Journal of Algebra},
volume = {577},
pages = {61-73},
year = {2021},
issn = {0021-8693},
doi = {https://doi.org/10.1016/j.jalgebra.2021.03.002},
url = {https://www.sciencedirect.com/science/article/pii/S0021869321001204},
author = {Alapan Mukhopadhyay and Karen E. Smith}
}

@book{berthelot1974cohomologie,
  title={Cohomologie cristalline des sch{\'e}mas de caract{\'e}ristique $p > 0$},
  author={Berthelot, Pierre},
  year={1974},
  publisher={Springer-Verlag},
  series={Lecture Notes in Mathematics},
  volume={407},
  address={Berlin, Heidelberg},
  isbn={978-3-540-06850-1}
}

@article{DattaOlanderRelFrobIso,
      title={A characteristic $p$ analog of formal lifting properties}, 
      author={Rankeya Datta and Noah Olander},
      year={2025},
      Journal={arXiv:2512.00648},
      eprint={2512.00648},
      archivePrefix={arXiv},
      primaryClass={math.AC},
      url={https://arxiv.org/abs/2512.00648}, 
}

@artile{CarvajalRojasStablerPristineMorphisms,
      title={On pristine morphisms}, 
      author={Javier Carvajal-Rojas and Axel Stäbler},
      year={2026},
      Journal={arXiv:2512.06063},
      eprint={2512.06063},
      archivePrefix={arXiv},
      primaryClass={math.AG},
      url={https://arxiv.org/abs/2512.06063}, 
}

@misc{MaPolstraFSing,
      title={{$F$}-singularities: a commutative algebra approach}, 
      author={Ma, Linquan and Polstra, Thomas},
      year={2025},
      url={https://www.math.purdue.edu/~ma326/F-singularitiesBook.pdf},
    note         = {available at \url{https://www.math.purdue.edu/~ma326/F-singularitiesBook.pdf}}
}

@misc{Bhatt_imperfect_trivial_cc,
  author = {Bhatt, Bhargav},
  title = {An imperfect ring with a trivial cotangent complex},
  howpublished = {},
  year = {2012}, 
  url = {https://www.math.ias.edu/~bhatt/math/trivial-cc.pdf},
  note = {available at \url{https://www.math.ias.edu/~bhatt/math/trivial-cc.pdf}}
}

\end{document}